\documentclass[article]{siamart1116}
\usepackage{amsmath}
\usepackage{float}
\usepackage{diagbox}
\usepackage{bbold}

\newtheorem{thm}{Theorem}
\newtheorem{lem}{Lemma}
\newtheorem{cor}{Corollary}
\newtheorem{rmk}{Remark}

\newcommand{\be}{\begin{equation}}
\newcommand{\ee}{\end{equation}}

\newcommand{\dx}{\Delta x}
\newcommand{\dt}{\Delta t}

\newcommand{\m}[1]{\mathbf{#1}}
\newcommand{\mA}{\m{A}}

\newcommand{\mD}{\m{D}}

\newcommand{\mR}{\m{R}}

\newcommand{\mI}{\m{I}}

\renewcommand{\v}[1]{\boldsymbol{#1}}

\newcommand{\vc}{\v{c}}
\newcommand{\vP}{\v{P}}
\newcommand{\ve}{{\mathbb{1}}}

\newcommand{\vy}{\v{y}}

\newcommand{\vt}{\v{t}}

\newcommand{\ste}{\boldsymbol{\tau}}

\renewcommand{\v}[1]{\mathbf{#1}}

\title{Explicit and implicit error inhibiting schemes with post-processing}

\author{%
Adi Ditkowski\thanks{School of Mathematical Sciences, Tel Aviv University, Tel Aviv 69978, Israel. 
Email: adid@post.tau.ac.il} \and
Sigal Gottlieb\thanks{Mathematics Department, University of Massachusetts Dartmouth, 285 Old Westport Road,
North Dartmouth MA 02747. Email: sgottlieb@umassd.edu} \and
Zachary J. Grant\thanks{Department of Computational and Applied Mathematics, Oak Ridge National Laboratory, Oak Ridge TN 37830. Email: grantzj@ornl.gov} 
}

\begin{document}
\maketitle


\bibliographystyle{siam}

\begin{abstract} 
Efficient high order numerical methods for evolving the solution of an ordinary differential
equation are widely used. 
The popular Runge--Kutta methods, linear multi-step methods, and more broadly general linear methods,
 all have a global error that is completely determined  by analysis of the local truncation error.
In prior work we  investigated the interplay between
the local truncation error and the global error  to construct {\em error inhibiting schemes}
that control  the accumulation of the local truncation error over time,  resulting in a global
error that is one order higher than expected from the local truncation error. 
In this work we extend our error inhibiting framework introduced in \cite{EISpaper1} to include a broader class of
time-discretization methods that allows an exact computation of the leading error term, which can then be
 post-processed to obtain a solution that is two orders higher than expected from
 truncation error analysis. We define sufficient conditions that result in a desired form of the error and describe
  the construction of the post-processor. A number of new explicit and implicit methods that have this property are
  given and tested on a variety of  ordinary and partial  differential equations.
  We show that these methods provide a solution that is two orders higher than expected from
  truncation error analysis alone.
 \end{abstract}

\section{Introduction\label{sec:intro}}

Efficient high order time evolution methods are of interest in many simulations, particularly when evolving in time
a system of ordinary differential equations (ODEs) resulting from a semi-discretized partial differential equation.
In this work we consider an approach to developing methods that are of higher order than expected from 
truncation error analysis. We first present some background on numerical methods for ODEs and define all the relevant terms.

We consider numerical solvers for  ODEs. Without loss of generality, we can consider the autonomous
equation of the form
\begin{eqnarray}\label{ODE}
& & u_t =   F(u)  \;,\;\;\;\;\;  t \ge t_0  \\
& & u(t_0) =u_0 .\; \nonumber
\end{eqnarray}
The most basic of these numerical solvers is  the  forward Euler method
\[ v^{n+1} = v^n + \Delta t F(v^n) \; , \] 
where $v_n$ approximates the exact solution $v^n \approx u(t_n)$ at some time $t_n = t_0 + n \dt$.
The forward Euler method has local truncation error $LTE^n$ and approximation error  $\tau^n$ at any given time 
$t_n$ defined by
\[ \tau^n  = \dt LTE^n =   u(t^{n-1})  +  \Delta t F( u(t_{n-1}) )  -   u(t_{n}) \approx O(\Delta t^2) \]
and a global error  which is first order accurate
\[ E^n = v^n - u(t_n)   \approx O(\Delta t) .\]
We note that while we follow the convention in   \cite{gustafsson1995time,quarteroni2010numerical,AllenIsaacson,IsaacsonKeller,Sewell}
and define the local truncation error as the normalized approximation error, 
\[   LTE^n = \frac{1}{\dt}  \tau^n ,\]
an alternative notation in common use is to  define the  local truncation error to be the approximation error. 
Using our notation, the local truncation error and the global error are of the same order for the forward Euler method.
Using the alternative notation, the local truncation error is one order higher than the global error.

If we want a more accurate method than forward Euler, 
we can add steps and define a linear multistep method \cite{butcher2008numerical}
\begin{equation}\label{multistep}
v^{n+1} \,=  \sum_{j=1}^s \alpha_j\,  v^{n+1-j}  + \, \Delta t  \sum_{j=0}^s \beta_j F(v^{n+1-j} ),
\end{equation}
or we  can use multiple stages as in  Runge--Kutta methods \cite{butcher2008numerical}:
\begin{eqnarray}\label{RK_Butcher}
y_i  & = & v^n+ \Delta t  \sum_{j=1}^m a_{ij}  F\left(   y_{j}  \right)\; \; \;  \mbox{for} \; \; j=1,...,m \nonumber \\
v^{n+1}& =& v_n + \Delta t \sum_{j=1}^m b_j F\left(  y_{j}\right) .
\end{eqnarray}

It is also possible to combine the approaches above, by starting with a Runge--Kutta method of the form \eqref{RK_Butcher}
but including older time-steps as in linear multistep methods \eqref{multistep}.
These are the general linear methods described in \cite{butcher1993a,JackiewiczBook} and may be written in the form
 \begin{eqnarray} \label{GLM}
y_i & = &  \sum_{j=1}^s \alpha_{ij} v^{n-j+1} +  \Delta t \sum_{j=1}^{s-1} \hat{a}_{ij} F( v^{n-j+1}) 
+\Delta t \sum_{j=1}^m a_{ij} F( y_j)  \nonumber \\
&& \hspace{3in}  \mbox{for} \; \; j=1,...,m \nonumber \\
v^{n+1} & =& \sum_{j=1}^s \omega_{j} v^{n-j+1} +
 \Delta t \sum_{j=1}^{s-1} \hat{b}_{j} F( v^{n-j+1}) 
+ \Delta t \sum_{j=1}^m b_{j} F( y_j) \; .
\end{eqnarray}
In all these cases, we aim to increase the order of the local truncation
error and therefore to increase the order of the global error.

The relationship between the local truncation error and the global error is well-known.
The Dahlquist Equivalence Theorem \cite{Suli2003} states that any zero-stable, 
consistent linear multistep method with  local truncation error $LTE^n = O(\dt^p)$ will have global error $O(\dt^p)$, 
provided that the solution  has at least $( p + 1)$ smooth derivatives. (Note that this statement depends on
the normalized definition of the local truncation error where $LTE^n = \frac{1}{\dt} \ste^n$. If the non-normalized definition is used,
then the global order is one lower than the local truncation error). 
   Indeed, all the familiar schemes for numerically solving  ODEs
have global errors that are  of the same order as their local truncation errors.
In fact, this relationship between local truncation error and global error is common not only  for ODE solvers, 
but is typically seen in   other fields in numerical mathematics, such as for finite difference schemes for partial 
differential equations (PDEs) \cite{gustafsson1995time,quarteroni2010numerical}.
  
It is, however, possible to devise schemes that have global errors that are {\em higher order} than the local truncation errors.
In particular, it was shown in  \cite{ditkowski2015high} that  finite difference schemes for PDEs can be constructed such that 
their convergence rates  are higher than the order of the  truncation errors. 
A similar approach was adopted in \cite{EISpaper1} for time evolution methods, where  we described the conditions under which
general linear methods (GLMs)  can  have global error is {\em one order higher} than the normalized local truncation error, and devise
a number of GLM schemes that feature this behavior.
These schemes achieve this higher-than-expected order by inhibiting the lowest order term in the local error from accumulating over time, 
and so we named them {\em  Error Inhibiting Schemes} (EIS). As we have since discovered, our EIS schemes have the same conditions as the  
quasi-consistent schemes of Kulikov, Weiner, and colleagues 
\cite{Kulikov2009,SoleimaniWeiner2017,WeinerSchmitt2009}.

In this paper, we extend the error inhibiting approach in two ways. First, we consider a broader formulation of the GLM that allows
for implicit methods as well as more general explicit methods than considered in \cite{EISpaper1}. Next, we show that by imposing
additional conditions on the methods we are able to precisely describe the coefficients of the error term and we then devise a post-processing
approach that removes this error term, thus obtaining a global error that is {\em two orders higher} than the normalized local truncation error.
We proceed to devise error inhibiting methods using this new EIS approach, and investigate their 
linear stability properties, strong stability preserving (SSP) properties, and computational efficiency.
Finally, we test these methods on a number of numerical examples to demonstrate their enhanced accuracy and, 
where appropriate, stability properties. In Appendix A we present an alternative approach to understanding the 
growth of the errors and describing the  coefficients of the error term. Note that the analysis in this paper is performed for the scalar 
equation \eqref{ODE}, but can be easily extended to the vector case, as shown in Appendix B.

\subsection{Preliminaries}
We begin with a scheme of the form
\begin{equation} \label{EISmethod}
V^{n+1} = \mD V^n +  \Delta t \mA F( V^n) +  \Delta t \mR F( V^{n+1}) , 
\end{equation}
where $V^n$ is a  vector of length $s$ that contains the numerical solution
at times  $\left( t_n+ c_j \Delta t \right)$ for $j=1,\ldots,s$:
\begin{equation}\label{multistep_v}
V^n = \left(v(t_n+ c_1 \Delta t)) ,  v(t_{n} + c_2 \Delta t)  , \ldots ,v(t_{n} + c_{s} \Delta t )  \right )^T .
\end{equation}
The function $F(V^n)$ is defined as the component-wise function evaluation on the vector $V^n$:
\begin{equation}\label{multistep_F}
F(V^n) = \left( F( v(t_{n} + c_1 \Delta t) ) , F( v(t_{n} + c_2 \Delta t) ),  \ldots , F(v(t_{n} + c_{s}  \Delta t ) ) \right )^T .
\end{equation}
For convenience,  $c_1 \leq c_2 \leq ... \leq c_s$. Also, we pick $c_s =0$ so that the final element in the vector 
approximates the solution at time $t_n$. This notation was chosen to match the notation widely used in
the field of general linear methods. To start these methods, we need to build an initial vector: to do so we 
artificially redefine the time-grid so that $t_0 :=t_0-  c_1 \Delta t$. Now the first element in the initial solution 
vector $V^0$  is $v(t_0+ c_1 \Delta t) = u_0$ and the remaining elements  $v(t_0+ c_j \Delta t) $
are computed using a highly accurate  one-step  method. Note that we consider only the case where $\dt$ is fixed.

\begin{rmk} The form  \eqref{EISmethod} includes  implicit schemes, as 
$V^{n+1}$ appears on both sides of the equation. However, if 
$\mR$ is strictly lower triangular the scheme is, in fact,  explicit.
If $\mR$ is a diagonal matrix then we can compute each element of the vector $V^{n+1}$
concurrently.
\end{rmk}

We define the corresponding exact solution of the ODE \eqref{ODE}:
\begin{equation}\label{multistep_u}
U^n = \left(  u(t_{n} + c_1 \Delta t)  , u(t_{n} + c_2 \Delta t) ,  \ldots ,u(t_{n} + c_{s}  \Delta t )  \right )^T ,
\end{equation}
with $F(U^n)$ defined similarly.

The {\em global error} of the method is defined as the difference between the
vectors of the exact and the numerical solutions at some time $t^n$
\begin{equation}\label{E-globalerror}
E^n = V^n - U^n .
\end{equation}

For the method  \eqref{EISmethod}, we define the normalized local truncation error $LTE^n$ 
and approximation error $ \ste^n$ by 
\begin{equation} \label{localerrors}
 \Delta t \; LTE^n = \ste^n = 
\left[ \mD U^{n-1} +  \Delta t \mA F( U^{n-1}) +  \Delta t \mR F( U^{n}) \right]  -U^{n} .
 \end{equation}
 Taylor expansions on \eqref{localerrors} allow us to compute these terms
\begin{equation} 
 \Delta t \; LTE^n = \ste^n = \sum_{j=0}^{\infty} \ste^n_j  \Delta t^{j}   =
 \sum_{j=0}^{\infty} \ste_j  \Delta t^{j}  \left. \frac{d^j u} {dt^j} \right|_{t=t_n}   
  \end{equation}
where the term $ \left. \frac{d^j u} {dt^j} \right|_{t=t_n} $ is the $j$th derivative of the solution at time $t= t_n$,
the value $\Delta t^{j} $ is the time-step raised to the power $j$, and the truncation error vectors $\ste_j$ 
are given by:
\begin{subequations} \label{tau_j_def}
\begin{eqnarray} 
\ste_0 & = &  \left(\mI - \mD\right) \ve \\
\ste_j &=&  \frac{1}{(j-1)!} \left(\frac{1}{j} \mD (\vc-\ve)^j  +\mA (\vc-\ve)^{j-1} +  \mR \vc^{j-1} -
 \frac{1}{j} \vc^j    \right) \\
 &&  \hspace{2.65in} \mbox{for j=1,2, . . .} \nonumber
  \end{eqnarray}
  \end{subequations}
Here,  $\vc$ is the vector of abscissas $(c_1, c_2, . . . , c_{s})^T$ and $\ve = (1,1, . . . , 1)^T$ is the vector of ones.
Terms of the form $\vc^j$ are understood component-wise  $\vc^j = (c_1^j, c_2^j, . . . , c^j_{s})^T$.

If the truncation error vectors $\ste_j$ satisfy the order conditions $\ste_j =0$ for $j=0, . . .,p$
then we have  local truncation error $  LTE^n  = O(\Delta t^{p} )$ and (correspondingly) approximation error 
 $\ste^n =  O(\Delta t^{p+1} )$.  
 In such a case  we typically observe a global order  of $E^n = O(\Delta t^{p} )$ at any time $t^n$. 
 Note once again that our definition of the  local truncation error $ LTE^n  $ is the normalized definition.
 According to this definition, we sww that we expect the global error to be of the {\em same}
 order as the local truncation error $LTE^n$. If the local truncation error $ LTE^n  $ is not normalized, then it
 is equal to $\ste^n$, and we expect the global error to be  {\em one order lower }
than $\ste^n$. We reiterate that we use the local truncation error $LTE^n$ to indicate the normalized definition,
and the approximation error $\ste^n$ to indicate the non-normalized definition.
 
 As shown in \cite{EISpaper1}, if we add conditions on $\mD, \mA$ (and $\mR$),
 we can find methods that exhibit a global order of $E^n = O(\Delta t^{p+1} )$ 
even though they only satisfy $\ste_j =0$ for $j=0, . . .,p$. 
In addition, as we show in Section \ref{sec:iEIS}, certain conditions on $\mD, \mA,\mR$ allow us to
precisely understand the form of the $\Delta t^{p+1} $ term in the global error, so we can design 
a postprocessor that gives us a numerical solution which has an error  of order $O(\Delta t^{p+2}) $, 
as we show in Section \ref{sec:postproc}.

 \section{Review of error inhibiting and related schemes}\label{sec:review}
The Dahlquist Equivalence Theorem \cite{Suli2003} states that any zero-stable, 
consistent linear multistep method with truncation error $O(\dt^p)$ will have global error $O(\dt^p)$, 
provided that the solution  has at least $( p + 1)$ smooth derivatives. 
All the one step methods and linear multistep methods  that are commonly used 
feature a global error  that is  of the same order as the normalized local truncation error. 
This is so common that the  order of the method is typically defined solely by the order conditions 
derived by Taylor series analysis of  the local truncation error. 
However, recent work \cite{EISpaper1} has shown that one can construct  general linear methods 
so that the accumulation of the local truncation error over time is controlled.
These schemes feature a global error that is one order higher than the local truncation error. 
In this section we review prior work on such error inhibiting schemes by Ditkowski and Gottlieb \cite{EISpaper1},
and show that similar results were obtained  in the work of Kulikov, Weiner, and colleagues 
\cite{Kulikov2009,SoleimaniWeiner2017,WeinerSchmitt2009}.
This body of work serves as the basis for our current work which will be presented in Section 3.

In \cite{EISpaper1}, Gottlieb and Ditkowski considered numerical methods of the form
\begin{equation} \label{EIS_multistep_2}
V^{n+1} = \mD V^n + \dt \mA F(V^n)
\end{equation}
(i.e. \ref{EISmethod} with $\mR=0$),
where  $\mD$ is a diagonalizable rank one matrix, whose non-zero eigenvalue is equal to one and its corresponding eigenvector is 
$  \left( 1, \ldots, 1\right)^T $. They showed that  if the method satisfies the order conditions
\begin{eqnarray} \label{OC}
 \ste_j =  0 , \; \; \; \mbox{ for} \; \; \; j=0, . . ., p 
 \end{eqnarray}
and the {\em error inhibiting condition}
\begin{eqnarray} \label{EIScondition}
 \mD\ste_{p+1}=0  
 \end{eqnarray}
hold, then the resulting global error satisfies $ E^n= O\left( \Delta t^{p+1} \right) . $
Our work in \cite{EISpaper1} showed how condition \eqref{EIScondition} mitigates the accumulation of  the truncation error,  
so we obtain a global error that is one order  {\em higher}   than predicted by the order conditions  
that describe the local truncation error. Furthermore, we developed several 
block one-step methods and demonstrated in numerical examples on nonlinear problems that these
methods have global error that is one order higher than the local truncation error analysis predicts.

As we later found out, although the underlying approach is different, the conditions  in \cite{EISpaper1}
are along the  lines of the theory of quasi-consistency  first introduced by Skeel in 1978 \cite{Skeel1978}. This work was then generalized and
further developed by Kulikov for Nordsiek methods in \cite{Kulikov2009}, as well as  Weiner and colleagues \cite{WeinerSchmitt2009}.
In \cite{WeinerSchmitt2009}, a theory of superconvergence for  explicit two-step peer methods was defined, 
and in \cite{SoleimaniWeiner2017} the theory was extended to implicit methods. 
In these papers, the authors showed that if the method satisfies order conditions \eqref{OC} and conditions that are equivalent to the EIS conditions
in \cite{EISpaper1} then the global error satisfies $ E^n= O\left( \Delta t^{p+1} \right) $. 

In \cite{KulikovWeiner2010}, \cite{KulikovWeiner2012}, and \cite{KulikovWeiner2018}, 
the authors show that by requiring the order conditions \eqref{OC} and EIS condition \eqref{EIScondition}
as well as additional conditions
\begin{subequations}  \label{KW2010}
\begin{eqnarray}
\mD\ste_{p+2} & =& 0 \\ 
\mD \mA \ste_{p+1} & =& 0 \\ 
\mD \mR \ste_{p+1} & =& 0 
\end{eqnarray}
\end{subequations}
the resulting global error satisfies $E^n= O\left( \Delta t^{p+1} \right) $
but, in addition, the form of the first surviving term in the global error (the vector multiplying $\Delta t^{p+1} $) is known
explicitly, and is leveraged for error estimation. 
While  the error inhibiting condition \eqref{EIScondition} requires the $p+1$ truncation error vector $\ste_{p+1} $ to live in the nullspace of 
$\mD$ the additional conditions  \eqref{KW2010} require  that the $p+2$ truncation error vector $\ste_{p+2} $ also lives in the nullspace of 
$\mD$, and that the  $p+1$ truncation error vector $\ste_{p+1} $ lives in the nullspace of $\mD \mA$ and $\mD \mR$.
These conditions inhibit the buildup of the truncation error so that we not only get a solution that is of order $\Delta t^{p+1} $, but also
 control the vector multiplying this error term. 

In this work we  can show that under {\em less restrictive conditions} than \eqref{KW2010} we can explicitly compute the 
form of the first surviving term in the global error. Furthermore, having this error term explicitly defined enables us to define
a postprocessor that allows us to extract a final-time solution that is accurate to two orders higher than predicted by truncation error
analysis alone.  We proceed to demonstrate these two facts and design a post-processor that allows us to extract a solution of $O(\dt^{p+2})$.

 \section{Designing explicit and implicit error inhibiting schemes with post-processing} \label{EISppTheory}
In this section we consider the improved error inhibiting schemes and the associated post-processor that allow us to recover order $p+2$ from  a scheme that would otherwise be only $p$th order accurate. 
In  Subsection \ref{sec:iEIS}, we show that the method must satisfy additional conditions so that the form of the final error is controlled and can be post-processed to extract higher order. In Subsection \ref{sec:postproc}, we show how this post-processing is done. Before we begin, we make some observations in Subsection
\ref{sec:essentials} that will be essential for Subsection \ref{sec:iEIS}. We note that our discussion in this section focuses, for simplicity of presentation, on a scalar 
function $F$. The vector case follows directly without difficulty, but the notation is more messy.  We  briefly discuss the extension to the vector case in Appendix \ref{sec:appendixB}.

\subsection{Essentials} \label{sec:essentials}
In this subsection we present two observations that will be used in the next subsection. 
The first observation uses the smoothness of the time evolution operator to bound the growth of the error 
at each step:
\begin{lem}  \label{lemma2}
The scheme \eqref{EISmethod} can be written in the form:
\begin{eqnarray} \label{EISmethod_10}
V^{n+1}  \,= \, \left( I -  \Delta t \mR F  \right)^{-1} 
\left(  \mD  +  \Delta t \mA F \right) V^n \, \equiv \, Q^n V^n
\end{eqnarray}
If the order conditions are satisfied to $j=0,...., p$, the scheme  \eqref{EISmethod}  (or equivalently \eqref{EISmethod_10})  is zero--stable, the operator $\left( I -  \Delta t \mR F  \right)^{-1} $ is bounded  and the (nonlinear) operator $Q^n$ is Lipschitz continuous in the sense that there is some constant $L>0$ such that
\begin{eqnarray} \label{EISmethod_20}
\left \|  Q^n V^n - Q^n U^n \right \|  \, \leq \, L\, \left \|  V^n -  U^n \right \|
\end{eqnarray}
 then the error accumulated in one step gets no worse than $O(\dt^{p+1}) $ i.e.
\begin{eqnarray} \label{EISmethod_30}
\left \| E^{n+1} \right\|   =  O \left(  \|  E^{k} \|  \right) ) + O( \Delta t^{p+1} ). 
\end{eqnarray}
\end{lem}

\begin{proof} 
The exact solution satisfies the equation
\begin{equation} \label{EISmethod_Lemma1_10}
 U^{n+1}  \,= \,  Q^n U^n \, - \,  \left( I -  \Delta t \mR F  \right)^{-1}  \ste^{n+1}  \,. 
\end{equation}
By subtracting  \eqref{EISmethod_Lemma1_10} from  \eqref{EISmethod_10}  we obtain
\begin{align*} \label{EISmethod_Lemma1_20}
\left \| E^{n+1} \right\| & =\, \left \| V^{n+1} - U^{n+1}\right\|  \le   \left \| Q^n V^{n} - Q^n U^{n}\right\|  + 
   \left \| \left( I -  \Delta t \mR F  \right)^{-1}  \ste^{n+1}  \right\| 
  \\  \nonumber 
& \le   L \left \| V^{n} - U^{n}\right\| \, + \,  O( \Delta t^{p+1} )  \\  \nonumber 
&=  O \left(  \|  E^{k} \|  \right) ) + O( \Delta t^{p+1} ). 
\end{align*}
\end{proof}

\noindent The next observation will be needed for the expansion of $F(U^n + E^n)$. In the following, we assume that $F$ is a scalar function operating on the 
vector $V^n=U^n + E^n$. The extension of this to the case where $F$ is a vector function is straightforward, and will be addressed in Appendix \ref{sec:appendixB}.
\begin{lem}  \label{lemma1}
Given a smooth enough function $F$, we have 
\begin{equation}
F(U^n + E^n)   =  F(U^n)  + F_u^n E^n  + O(\dt)  O\left( \|E^n\| \right) ,
\end{equation}
where $F_u^n = \frac{\partial F}{\partial u}(u(t_n))$.
\end{lem}
\begin{proof}
We expand
\begin{align*}
F(U^n + E^n) & =    F(U^n) + 
\left(
\begin{array}{l}
F_u(u(t_n+ c_1 \Delta t)) e_{n+c_1} \\
F_u(u(t_n+ c_2 \Delta t)) e_{n+c_2} \\
\vdots \\
F_u(u(t_n+ c_{s} \Delta t)) e_{n+c_{s}} \\
\end{array}
\right)  + O\left( \|E^n\|^2 \right),
\end{align*}
where the error vector is $E^n = \left( e_{n+c_1} , e_{n+c_2} , . . . , e_{n+c_s} \right)^T$.

Using the definition $F_u^{n+c_j} = \frac{\partial F}{\partial u}(u(t_n+ c_j \Delta t)) $ we  re-write this as
\begin{align*}
F(U^n + E^n) 
& = &   F(U^n) + 
\left(
\begin{array}{cccc}
F_u^{n+ c_1 }  & 0 & \cdots & 0 \\
0 & F_u^{n+ c_2 } & \cdots & 0 \\
\vdots  & \vdots & \ddots & \vdots \\
0 & 0 & \cdots &  F_u^{n+ c_{s} } \\
\end{array}
\right) E^n + O\left( \|E^n\|^2 \right) .
\end{align*}
Each term can be expanded as
\begin{equation*}
 F_u^{n+ c_j  }  =  F_u(u(t_n+ c_j \Delta t))  = F_u(u(t_n)) +  c_j \Delta t F_{uu}(u(t_n))  + O(\dt^2)\\
 \end{equation*}
 so that we can say that 
 \begin{equation*}
 F_u^{n+ c_j  }  =  F_u(u(t_n+ c_j \Delta t))  = F_u(u(t_n)) +  C \Delta t   + O(\dt^2)\\
 \end{equation*}
so that
\begin{align*}
F(U^n + E^n)  & =    F(U^n)  + \left( F_u^n I + C \dt + O(\dt^2) \right) E^n + O\left( \|E^n\|^2 \right)  \\
& =  F(U^n)  + F_u^n E^n  + C \dt E^n  +   O(\dt^2)  O\left( \|E^n\| \right) + O\left( \|E^n\|^2 \right) .\\
& =  F(U^n)  + F_u^n E^n  + O( \dt) O\left( \|E^n\| \right) .\\
\end{align*}
\end{proof}

\begin{cor} \label{cor1}
If the error is of the form $E^n = O(\dt^{p+1})$, then
\begin{equation}
 F(U^n + E^n)  = F(U^n)  + F_u^n E^n  + O \left( \dt^{p+2}  \right).
 \end{equation}
\end{cor}

\subsection{Improved error inhibiting schemes that have $p+1$ order} \label{sec:iEIS}

In this section we consider a method  of the form   \eqref{EISmethod}, 
where $\mD$ is a rank one matrix  that satisfies the consistency condition $\mD \ve =\ve$,
and $\mD$, $\mA$ and $\mR$ are matrices that satisfy the order conditions
\begin{equation}  \label{OC-tau}
\ste_j =  0 , \; \; \; \mbox{ for} \; \; \; j=1, . . ., p ,
\end{equation}
and the EIS+ conditions
\begin{subequations}  \label{EIS2conditions}
\begin{eqnarray}
\mD\ste_{p+1}=0  \label{con1} \\
\mD\ste_{p+2}=0  \label{con2}\\
\mD(\mA+\mR)   \ste_{p+1}=0 \label{con3}.
\end{eqnarray}
\end{subequations}

We initialize the method with a numerical solution vector $V^0$ that is accurate
 enough so that we ensure that the error is negligible. This means that 
taking one step forward using   \eqref{EISmethod} has no accumulation error, only the truncation error:
\[ V^1  =  U^1 +\ste^1 = U^1 +  \Delta t^{p+1}  \ste^1_{p+1} +  \Delta t^{p+2} \ste^1_{p+2}  + O( \Delta t^{p+3} ) ,\]
and
\[ E^1  =  \Delta t^{p+1}  \ste^1_{p+1} +  \Delta t^{p+2} \ste^1_{p+2}  + O( \Delta t^{p+3} ) .\]
The conditions $\mD  \ste_{p+1} = 0$ and  $\mD  \ste_{p+2} = 0$ from \eqref{EIS2conditions} mean that
  \[ \mD E^1 = O( \Delta t^{p+3} ) .\] 
  
  From Lemma \ref{lemma2} we know that the error that accumulates in only one step is no worse than order 
  $\dt^{p+1}$ so that $O(E^{k+1}) = O(E^k) + O(\dt^{p+1})$ and thus we know that
\[ V^2  =   U^2 + O(\Delta t^{p+1} )  .\]
We use these facts about $V^1$ and $V^2$ to write:
\begin{eqnarray*}
 V^2 & = & \mD V^1 + \Delta t \mA F(V^1) + \Delta t \mR F(V^{2})  \\ 
& = & \mD \left(U^1 + E^1  \right) +  \Delta t \mA F\left(U^1 + E^1 \right) 
  + \Delta t \mR F\left(U^{2}+ E^2   \right) \\ 
& = & \mD U^1  +  \Delta t \mA F(U^1)  +  \Delta t \mR F(U^2)  +\mD E^1  \\
&  +&   \Delta t \mA  F_u^1 E^1 + \Delta t \mA  F_u^2 E^2  +O \left( \dt^{p+3}  \right) 
\end{eqnarray*}
Note that the  first three terms are, by the definition of the local truncation error
\[  \mD U^1 +  \Delta t \mA F(U^1) +  \Delta t \mR F(U^2)  = U^2 + \dt^{p+1}  \ste^2_{p+1} + 
\dt^{p+2} \ste^2_{p+2}  +  O(\dt^{p+3}),\]
so that
\[ V^2  =   U^2 + \ste^2 +  C \dt^{p+2} + O(\dt^{p+3}),\]
which means that \[ E^2 = \ste^2 +  C \dt^{p+2} + O(\dt^{p+3}).\]
Using this more refined understanding of the order term, we repeat the process above to write:
\begin{eqnarray*}
 V^2 & = & \mD \left(U^1 + E^1  \right) +  \Delta t \mA F\left(U^1 + E^1 \right) 
  + \Delta t \mR F\left(U^{2}+ E^2   \right) \\ 
   & = & \left[ \mD U^1 +  \Delta t \mA F(U^1) +  \Delta t \mR F(U^2) \right] +  \mD E^1 \\
   &+& \dt \mA E^1 F_u^1+ \dt \mR E^2 F_u^2    + C \dt^2 E^n  +   O(\dt^3)  O\left( \|E^n\| \right) \\
     & = & U^2 + \ste^2 + \dt \mA \ste^1 F_u^1 + \dt \mR \ste^2 F_u^2  + C \dt^{p+3} + O\left(  \dt^{p+4} \right) 
 %
 \end{eqnarray*}
 where $C$ is a constant that depends on the truncation error vectors and the values of higher derivatives of a solution, and
 whose value is different at each line.
This means that
 \begin{eqnarray} \label{basecase}
 E^2 &=& \dt^{p+1}  \ste_{p+1}^2 + \dt^{p+2}  \ste_{p+2}^2 
  +  \dt^{p+2} \left(  \mA + \mR \right)  \ste_{p+1}^2 F_u^2 \\
  &+&   2 C_2 \dt^{p+3}  + O(\dt^{p+4} )  ,  \nonumber
  \end{eqnarray}
  and so conditions \eqref{EIS2conditions} imply that 
   \[ \mD E^2  =  \tilde{C} \dt^{p+3} +   O(\dt^{p+4} ). \]

\smallskip

This precise form of $E^2$ can be extended to the error at any time-level $t_n$.
In the following theorem we show that the  conditions  \eqref{EIS2conditions} 
 allow us to determine  the precise form of the $\dt^{p+1}$ term in the global error $E^n$ 
resulting from the  scheme \eqref{EISmethod}. In the next section we use this precise form
to extract higher order accuracy by post-processing.

\begin{thm} 
For a method of the form   \eqref{EISmethod}, if we
 choose a rank one matrix $\mD$ such that the consistency condition $\mD \ve =\ve$ is satisfied, 
and matrices $\mA$ and $\mR$ such that the order conditions \eqref{OC-tau} are satisfied 
and furthermore the EIS conditions \eqref{EIS2conditions} hold, then the resulting global error at any time
$t_n$  has the form
\begin{equation} \label{Error_form}
E^n= \Delta t^{p+1}   \ste^n_{p+1} +  \Delta t^{p+2}   \ste^n_{p+2}  
+ \Delta t^{p+2} (\mA+\mR)  \ste^n_{p+1} F_u^n 
+ C_n n   \Delta t^{p+3} +  O( \Delta t^{p+3}),
\end{equation}
where $C_n$ is a constant.
\end{thm}

\begin{proof}  We showed above that the base case  \eqref{basecase} 
has the correct form \eqref{Error_form}.
We proceed by induction: assume that the numerical solution at time $t_k$ has an
error of the form \eqref{Error_form}:
 \[E^k = \Delta t^{p+1}  \ste^k_{p+1} +  \Delta t^{p+2} \ste^k_{p+2}  +  \Delta t^{p+2}  ( \mA  + \mR)  \ste^k_{p+1}  F_u(U^{k}) + C_k k \dt^{p+3}  + O( \dt^{p+3} ) . \]
  The form of the error  means that, using conditions \eqref{EIS2conditions}, we have
  \[\mD E^k =  C k \dt^{p+3}  + O( \dt^{p+3} ) .\]
 
We wish to show that under the conditions in the theorem,  $E^{k+1}$ has a similar form. We split
our argument  into two steps:  \\
\noindent{\bf (1) First step:} We know from Lemma \ref{lemma2} that the  
error that accumulates in only one step is no worse than order $\dt^{p+1}$ so that since
$ E^{k} =C  \Delta t^{p+1} + O( \Delta t^{p+1} )$ we know that we also have
$E^{k+1}  =  \tilde{C}  \Delta t^{p+1} + O( \Delta t^{p+1} ).$  Now we look at the evolution of the error from one time-step to the next:
 \begin{eqnarray*}
 V^{k+1} & = & \mD V^k + \Delta t \mA F(V^k) + \Delta t \mR F(V^{k+1})  \\ 
& = & \mD \left(U^k + E^k  \right) +  \Delta t \mA F\left(U^k + E^k \right) 
  + \Delta t \mR F\left(U^{k+1}+ E^{k+1} \right) \\ 
& = & \left[ \mD U^k +  \Delta t \mA F(U^k) +  \Delta t \mR F(U^{k+1}) \right]
   + \mD E^k \\
   & + &  \Delta t \mA F_u^k E^k +  \Delta t \mR F_u^{k+1} E^{k+1} 
   +  C \dt^2  O( \|E^k\| )  + O( \Delta t^{p+3} ). \\
   & = & U^{k+1} + \ste^{k+1} + C \Delta t^{p+2} + O(\Delta t^{p+3} )
  \end{eqnarray*}
 so that
\begin{eqnarray*}
E^{k+1}    & = &   \Delta t^{p+1}  \ste^{k+1}_{p+1} + \tilde{C} \Delta t^{p+2} +  O(\Delta t^{p+3} ).
   \end{eqnarray*}
   
\noindent{\bf (2) Second step:}
  This analysis allowed us to obtain a more precise understanding of the growth of the error 
  over one step from $V^k$ to $V^{k+1}$. The error is
  still of order $O(\Delta t^{p+1})$, but the leading term in the error is now explicitly defined. 
  With this new understanding of $V^{k+1}$,  we repeat the analysis:
   \begin{eqnarray*}
 V^{k+1} & = & \mD V^k + \Delta t \mA F(V^k) + \Delta t \mR F(V^{k+1})  \\ 
& = & \mD \left(U^k + E^k \right) +  \Delta t \mA F\left(U^k + E^k  \right) 
  + \Delta t \mR F\left(U^{k+1}+   E^{k+1} \right) \\
  & = & \left[ \mD U^k +  \Delta t \mA F(U^k) +  \Delta t \mR F(U^{k+1}) \right] +  \mD E^k \\
  &+&  \Delta t \mA E^k F_u^k +  \Delta t \mR E^{k+1} F_u^{k+1} + C \Delta t^2 O(\|E^k\|) + O(\dt^3) O( \|E^k\| )\\
  & = & U^{k+1} + \ste^{k+1} +  \mD E^k +  \Delta t \mA E^k F_u^k +  \Delta t \mR E^{k+1} F_u^{k+1} + C \Delta t^2 O(\|E^k\|)  + O(\Delta t^3) O( \|E^k\| ) \\
  & = & U^{k+1} + \ste^{k+1} + C k \dt^{p+3} +  \Delta t \mA E^k F_u^k +  \Delta t \mR E^{k+1} F_u^{k+1} +    O( \dt^{p+3} ).
   \end{eqnarray*}
  We now use what we know about $E^k$ from the base-case and $E^{k+1}$ from step 1:
     \begin{eqnarray*}
 V^{k+1} & = & U^{k+1} + \ste^{k+1} + C k \dt^{p+3} +   \Delta t \mA \left(  \Delta t^{p+1}  \ste^k_{p+1}   \right) F_u^k \\
 &+&   \Delta t \mR \left(    \Delta t^{p+1}  \ste^{k+1}_{p+1}  \right) F_u^{k+1} + \tilde{C} \Delta t^{p+3}  +   O( \dt^{p+3} ).\\
 & = & U^{k+1} + \ste^{k+1}  +   \Delta t^{p+2} \mA  \ste^k_{p+1}  F_u^k +  \Delta t^{p+2} \mR \ste^{k+1}_{p+1} F_u^{k+1} \\
 &+& C_k k \dt^{p+3}+  \hat{C}\Delta t^{p+3}   +   O( \dt^{p+3} ). \\
 & = & U^{k+1} + \Delta t^{p+1}   \ste^{k+1}_{p+1} +  \Delta t^{p+2}   \ste^{k+1}_{p+2}  +   \Delta t^{p+2} (\mA+\mR)   \ste^{k+1}_{p+1}  F_u^{k+1} \\
 &+&  C_k k \dt^{p+3}+  C \Delta t^{p+3}   +   O( \dt^{p+3} ), 
   \end{eqnarray*}
using the fact that  
 \[  \ste^{k+1}_{p+1}  =  \ste^{k}_{p+1} + O( \dt) \; \; \;  \mbox{and} \; \; \;  F_u^{k+1} = F_u^{k} + O( \dt),\]
 and that 
  \[ \ste^{k+1} = \Delta t^{p+1}   \ste^{k+1}_{p+1} +  \Delta t^{p+2}   \ste^{k+1}_{p+2} + \tilde{C} \dt^{p+3}+ O(\Delta t^{p+4}).\]
 Clearly, each time-step accumulates a few terms of the form $  \dt^{p+3}$, so that
   \begin{eqnarray*}
     V^{k+1}  &=& U^{k+1} + 
       \Delta t^{p+1}   \ste^{k+1}_{p+1} +  \Delta t^{p+2}   \ste^{k+1}_{p+2} 
      +  \Delta t^{p+2} \left( \mA + \mR \right) \ste^{k+1}_{p+1}   F_u^{k+1}  \\
      & +&  (C_k k + C )  \dt^{p+3} +O \left( \dt^{p+3} \right).
            \end{eqnarray*}
Setting $C_{k+1} = \max\{ C_k, C \}$ we obtain the desired result.
Note that the constants $C_k  $ are uniformly bounded, since they are constants that depend on some combinations of the truncation error vectors $\ste_j$, the 
higher order derivatives of $u$ at times $(t_0, t_n)$, and the bounded matrix operators in the problem.
  \end{proof}

\bigskip

The conditions $\ste_j =0$  for $j=1, . . ., p$ give us a method of order $p$, while the additional condition $\mD\ste_{p+1}=0$
allow us to get a $p+1$ order method, just as in previous EIS methods. 
Adding the conditions  $\mD\ste_{p+2}=0$ and $\mD(\mA+\mR)   \ste_{p+1}=0$ does not give us a higher order scheme. However, it allows us to 
control the growth of the error so that we can identify the error terms as in \eqref{Error_form}. 
This, in turn, allows us to  design a postprocessor that extracts order $p+2$, as we show in the next subsection.

\subsection{Post-processing to recover $p+2$ order} \label{sec:postproc}

In Theorem  1 we showed that the  at every time-step $t_n$ the
error $E^n$ has the form \eqref{Error_form}. The leading order term $\Delta t^{p+1}   \ste^n_{p+1}$
 can be filtered at the end of the computation in a post-processing stage as long as
 $\ste^n_{p+1}$ lies in a subspace which is ``distinct enough'' from the exact solution.

First we note that for  the error to be of order $\dt^{p+2}$  the exact solution of the ODE must have 
at least $p+2$ smooth derivatives, i.e. $u(t) \in C^{p+2}$. 
Therefore, up to an error of order $p+2$, we can expand the solution as a polynomial of degree $p+1$
 \begin{equation}\label{EIS_multistep_vc_error_10}
u(t) \,=\, \sum_{j=0}^{p+1} a_j (t-t_{\nu})^j \,+\, O(\dt^{p+2})
\end{equation}
at any point in the interval $[t_{\nu} - \dt, t_{\nu} + \dt]$.
It is then reasonable to expect  that the numerical solution, $V^n$, can be similarly expanded. 
Our post-processor is  built on these observations. 

To build the post-processor, we define the time vector
to be  the  temporal grid points in the last two computation steps:
 \begin{eqnarray}\label{EIS_multistep_vc_error_20}
\tilde{\vt} &=& \left (  t_{n-1} + c_1 \dt,       ...,   t_{n-1} + c_{s} \dt,  
 t_{n} + c_1 \dt,       ...,   t_{n} + c_{s} \dt
 \right )^T \\
 & = & \left ( 
\begin{array}{c}
\ve t_n + (\vc- \ve)\dt \\
\ve t_n + \vc \dt 
\end{array} 
\right ) . \nonumber
\end{eqnarray}
The last term in \eqref{EIS_multistep_vc_error_20} states that the vectors $\ve t_n + (\vc- \ve)\dt$ and 
$\ve t_n + \vc \dt$   are concatenated into one $2s$ long vector. 
Correspondingly, we let
 \begin{equation*} 
 \tilde{V}^n \,=\, \left ( 
\begin{array}{c}
V^{n-1} \\ 
V^{n}
\end{array} 
\right) \; \; \; \; \; \mbox{and} \; \; \; \; \; 
\tilde{U}^n
\,=\, \left ( 
\begin{array}{c}
U^{n-1} \\ 
U^{n}
\end{array} 
\right)
\end{equation*}
i.e., as for the temporal grid,  the numerical solutions $V^{n-1}$ and  $V^n$ are concatenated into 
one long vector, and the same is done for the exact solutions $U^{n-1}$ and  $U^n$. 
Similarly we define a concatenation of the leading truncation error terms
  \begin{equation}\label{EIS_multistep_vc_error_40}
\tilde{\ste}= \left(
\begin{array}{c}
 \ste_{p+1} \\ 
 \ste_{p+1}
\end{array} 
 \right ),
\end{equation}
where $\ste_{p+1}$ was defined in \eqref{tau_j_def}.
Note that by \eqref{Error_form},
\begin{equation}\label{EIS_multistep_vc_error_43}
\tilde{E}^n= \left(
\begin{array}{c}
    E^{n-1} \\
      E^n 
\end{array} 
 \right )  =   \tilde{\ste} \Delta t^{p+1} \left. \frac{d^j u} {dt^j} \right|_{t=t_n}  + O\left ( \Delta t^{p+2}\right ).
\end{equation}

\bigskip

Let $P_0(t-t_n), ..., P_{p+1}(t-t_n)$ be a set of polynomials of degree less or equal to $p+1$ such that 
\begin{equation}\label{EIS_multistep_vc_error_50}
\text{span} \left \{ P_0(t-t_n), ..., P_{p+1}(t-t_n) \right \} \,=\, \text{span} \left \{1, (t-t_n), ...,  (t-t_n)^{p+1} \right \}
\end{equation}
and that  the vectors $\vP_0, ..., \vP_{p+1}$ are the projections of $P_0(t-t_n), ..., P_{p+1}(t-t_n)$ onto the temporal grid points 
$\tilde{\vt} $.

Define the matrix $T \in {\mathbb C}^{2s \times 2s}$  as follows; The first column is $\tilde{\ste}$, the next
$p+2$ columns are  $\vP_0, ..., \vP_{p+1}$,  and the remaining  $2s-(p+3)$ columns  are selected such that 
$T$ can be inverted. A convenient way to accomplish this is to define the Vandermonde interpolation matrix
on the points in the $2s$-length vector $\tilde{\vt} $ and replace the highest order column 
(in {\sc Matlab} the first column) by $\tilde{\ste}$.

The post-processing filter is then given by 
 \begin{equation}\label{EIS_multistep_vc_error_60}
\Phi = T \, \text{diag} \left(0, \underbrace{1,...,1}_{p+2}, \underbrace{*,...,*}_{2s-(p+3)} \right)\,T^{-1}
\end{equation}
Where we select $*$ to be either 1 or 0, if it is desired to keep this subspace or eliminate it, respectively. 
Note that, by construction, the matrix $\Phi$ multiplied to the error vector $\tilde{E}^n$
eliminates the leading order term $\ste_{p+1}$ so that 
 \begin{equation}\label{EIS_multistep_vc_error_70}
\Phi \tilde{E}^n =    O\left ( \Delta t^{p+2}\right ).
\end{equation}
Also by construction, the matrix $\Phi$ includes the Vandermonde interpolation matrix up to order $p+1$
so that if it operates on a polynomial, it will replicate the polynomial exactly, up to order $p+1$.
Using  the polynomial expansion of $\tilde{U}^n$ assumed in \eqref{EIS_multistep_vc_error_10}
we then obtain
 \begin{equation}\label{EIS_multistep_vc_error_80}
\Phi \tilde{U}^n =   \tilde{U}^n +  O\left ( \Delta t^{p+2}\right ),
\end{equation}
and therefore the numerical solution obtained after the post-processing stage  is of order  $p+2$:
 \begin{equation}\label{EIS_multistep_vc_error_90}
\Phi \tilde{V}^n =   \Phi \left ( \tilde{U}^n + \tilde{E}^n \right ) =  \tilde{U}^n +  O\left ( \Delta t^{p+2}\right ).
\end{equation}

\begin{rmk}
We note that  is important that the leading term of the error does not live in the span of these polynomials.
In other words, we need $ \tilde{\ste}   \notin \textup{span} \left \{ \vP_0, ..., \vP_{p+1} \right \}$.
It is also important that the leading term of the error does not live too close to
the space spanned by these polynomials.
The coefficient in the $ O\left ( \Delta t^{p+2}\right )$ term of \eqref{EIS_multistep_vc_error_90}
depends on the norm of the matrix $\Phi$. This norm, in turn,
 is related to how far $ \tilde{\ste}  $ is  from
 its projection into the  subspace $ \textup{span} \left \{ \vP_0, ..., \vP_{s+1} \right \}$. 
 To avoid numerical instability,  it should be verified during the design stage
 that $ \| \Phi \|$   is not ``too large''.
\end{rmk}

Despite the verbose description above, the design of the post-processor matrix is computationally simple.
The following {\sc Matlab} code shows how this postprocessing matrix $P$ is computed:
\begin{verbatim}
m = 2;                             % # of intervals 
TAU = repmat(tau(:,p-1),m,1) ;     % truncation error vector
gp = c-(m-1:-1:0);                 % Gridpoints
S = (vander(gp(:)));               % Polynomial basis
S(:,1) = TAU;                      % Put TAU into basis
DD = diag([0,ones(1,length(S)-1)]);% Zero out truncation error
Phi = S*DD*inv(S);                 % Build post-processing matrix
\end{verbatim}

\smallskip

In the preceding discussion we assumed that there are enough points in time in each interval to produce a high
enough order polynomial: in other words, we assumed that
\[ 2 s \geq p+3.\] When this is not correct (e.g. when we have a scheme with $s=2$ that has
$p=2$, or a scheme with $s=3$ that has $p=4$) we need to use three repeats of each vector so that 
\[ 3 s \geq p+3.\]  In the {\sc Matlab} code  above this is equivalent to choosing $m=3$.

  \section{New error inhibiting schemes with post-processing}
Using MATLAB we coded an optimization problem \cite{EISgithub}
that finds the coefficients $\mD, \mA, \mR$ that satisfy the consistency and order conditions \eqref{OC-tau} 
and the error inhibiting conditions \eqref{EIS2conditions} while maximizing such properties as the 
linear stability region or the strong stability preserving coefficient \cite{SSPbook2011}.
In this section we present some  new methods obtained by the optimization code. 
These methods are error inhibiting (i.e. they satisfy \eqref{con1}) and one 
order higher can be extracted by postprocessing (i.e. they satisfy \eqref{con2} and \eqref{con3} as well).

\subsection{Explicit Methods}
In this section we present the explicit error inhibiting methods that can be post-processed. 
For each method we present the coefficients $ \mD, \mA, \mR$. We also give values of the 
abscissas and the truncation error vector that must be used to construct the 
postprocessor. We denote an explicit $s$-step method that satisfies the conditions \eqref{EIS2conditions} 
and can be postprocessed to order $p$ as eEIS+($s$,$p$). If, in addition, the method is
also strong stability preserving, we denote it eSSP-EIS+($s$,$p$).

\smallskip

\noindent{\bf Explicit error inhibiting method eEIS+(2,4):}
 This explicit two stage error inhibiting method is fourth
order if it is post processed, otherwise it gives third order accurate solutions.
The coefficients are:
\[
D =\frac{1}{2}   \left( \begin{array}{cc}
1  & 1 \\
1 & 1 \\
\end{array} \right), \; \; \; 
A =  \frac{1}{12} \left( \begin{array}{rr}  
- 7  & 17 \\
7  & - 5\\
\end{array} \right), \; \; \; 
R = \left( \begin{array}{cc}
0 & 0 \\
1  &0  \\
\end{array} \right).
\]
The abscissas are $c_1= -\frac{1}{3}, c_2 =0.$ The truncation error vector is 
$  \ste_3 = \frac{55}{324}\left( 1  , -1 \right)^T.$
The imaginary axis stability for this method is $(-0.6452,0.6452)$.
To post-process this method we  take a linear combination of several solutions
\[ \hat{v}^n = w_1 v^{n-2+c_1 } + w_2 v^{n-2} + w_3 v^{n-1+c_1} + w_4 v^{n-1} + w_5 v^{n+c_1} + w_6 v^{n} \]
with weights
\[ w_1= \frac{5}{108}, \; \; 
w_2= - \frac{7}{54},  \; \; 
  w_3= \frac{35}{108}, \; \;
  w_4= \frac{35}{108}, \; \;
   w_5=  \frac{7}{54},  \; \;        
     w_6= \frac{103}{108} .\]

These methods were constructed to maximize the linear stability region. 
In Figure \ref{stability_regions} we show the stability regions of these methods, and compare them 
to stability regions of the Adams-Bashforth linear multistep methods  with corresponding orders given in \cite{HairerNorsettWanner}.
When comparing these regions of stability, we need to keep in mind that while the linear multistep methods
require only one function evaluation per time step, the general linear methods have intermediate stages at
which the function evaluation is computed, which increases the computational cost to get to the final time.
For example, when looking at the   eEIS+(2,4)  method we need to compute two stages,
so that we require two function evaluations at each  time-step. 
For a fair comparison, the stability regions need to be scaled as well: for this reason, in Figure 
\ref{stability_regions}  the figure presented is visually scaled by the length of $V^n$, in that while for the  eEIS+(2,4) method
we show the axes as $(-1,1)$, for the fourth order Adams-Bashforth method the axes are $(-0.5,0.5)$.
Similarly, the axes for the eEIS+(3,6) are $(-1,1)$ while for the corresponding sixth order Adams-Bashforth method the axes are $(-0.33,0.33)$,
and for the eEIS+(5,7) the axes are $(-2.5,2.5)$ and the corresponding seventh order Adams-Bashforth method has axes $(-0.5,0.5)$.
(Note that for an explicit method,  we also may have some elements of  $F(V^{n+1})$ to compute at an intermediate stage, 
but these will be needed anyway at the next time-step so they do not add to the computational cost.)

\begin{figure}
 \includegraphics[width=.351\textwidth]{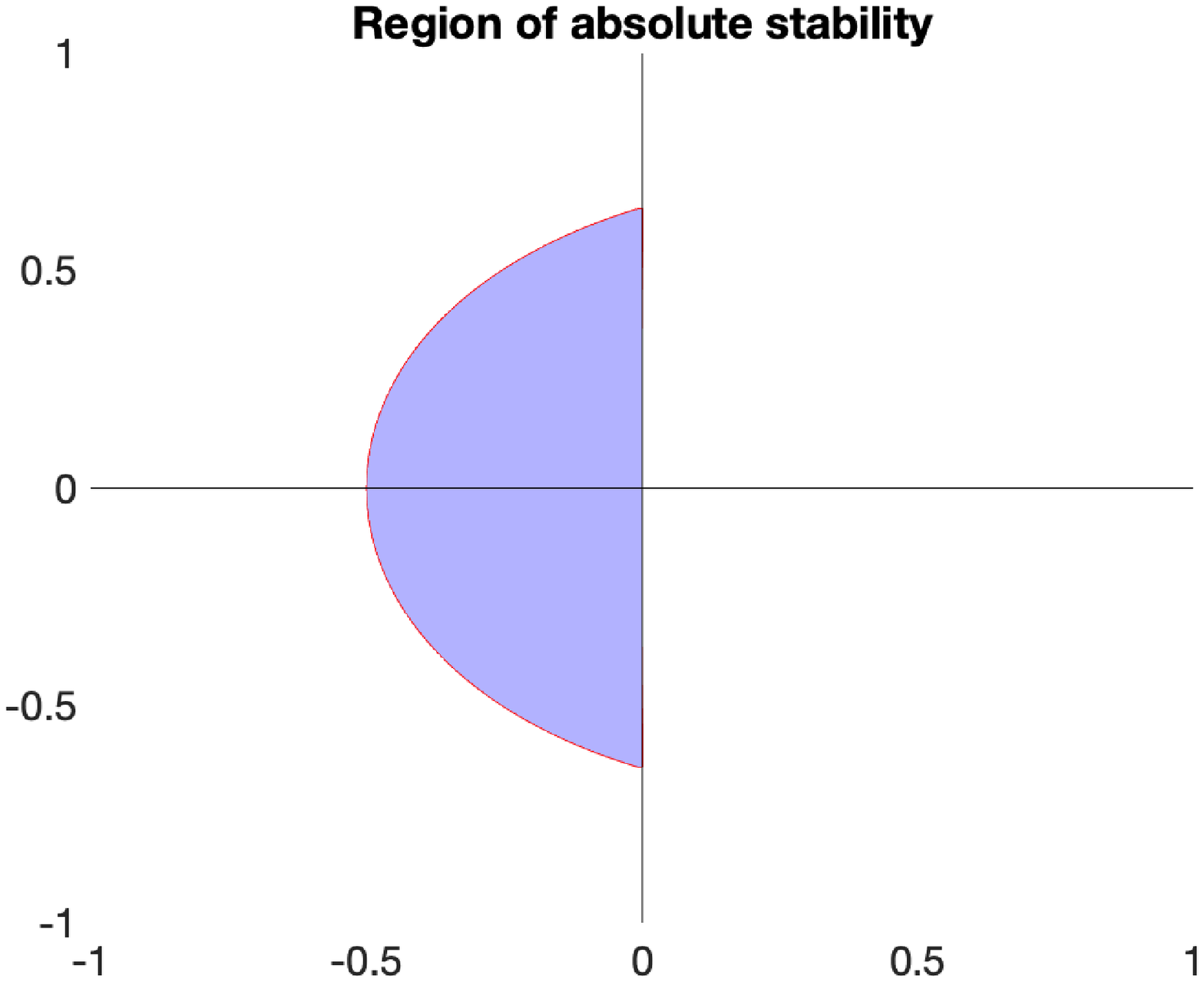} \hspace{-.2in}
  \includegraphics[width=.351\textwidth]{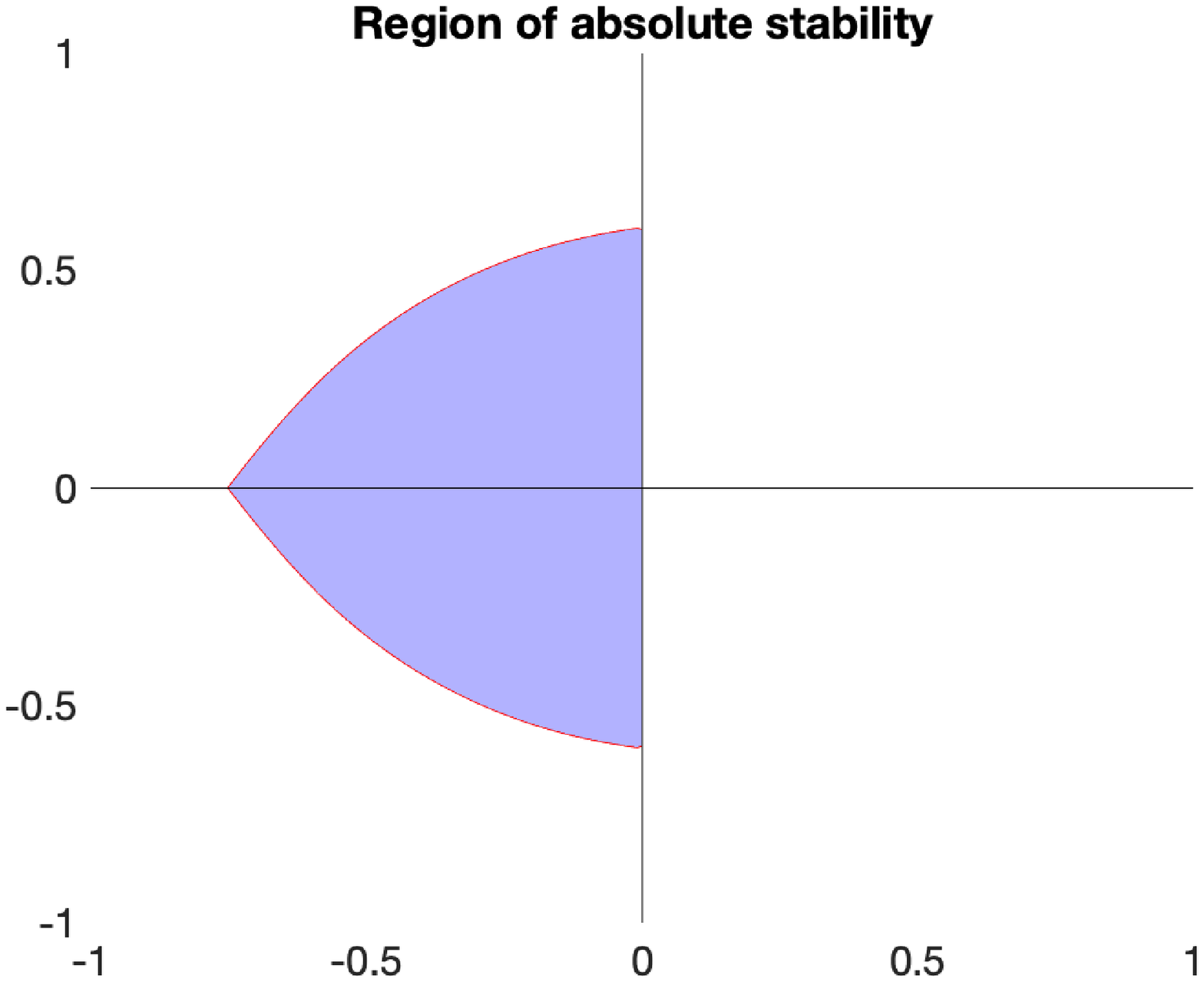} \hspace{-.2in}
   \includegraphics[width=.351\textwidth]{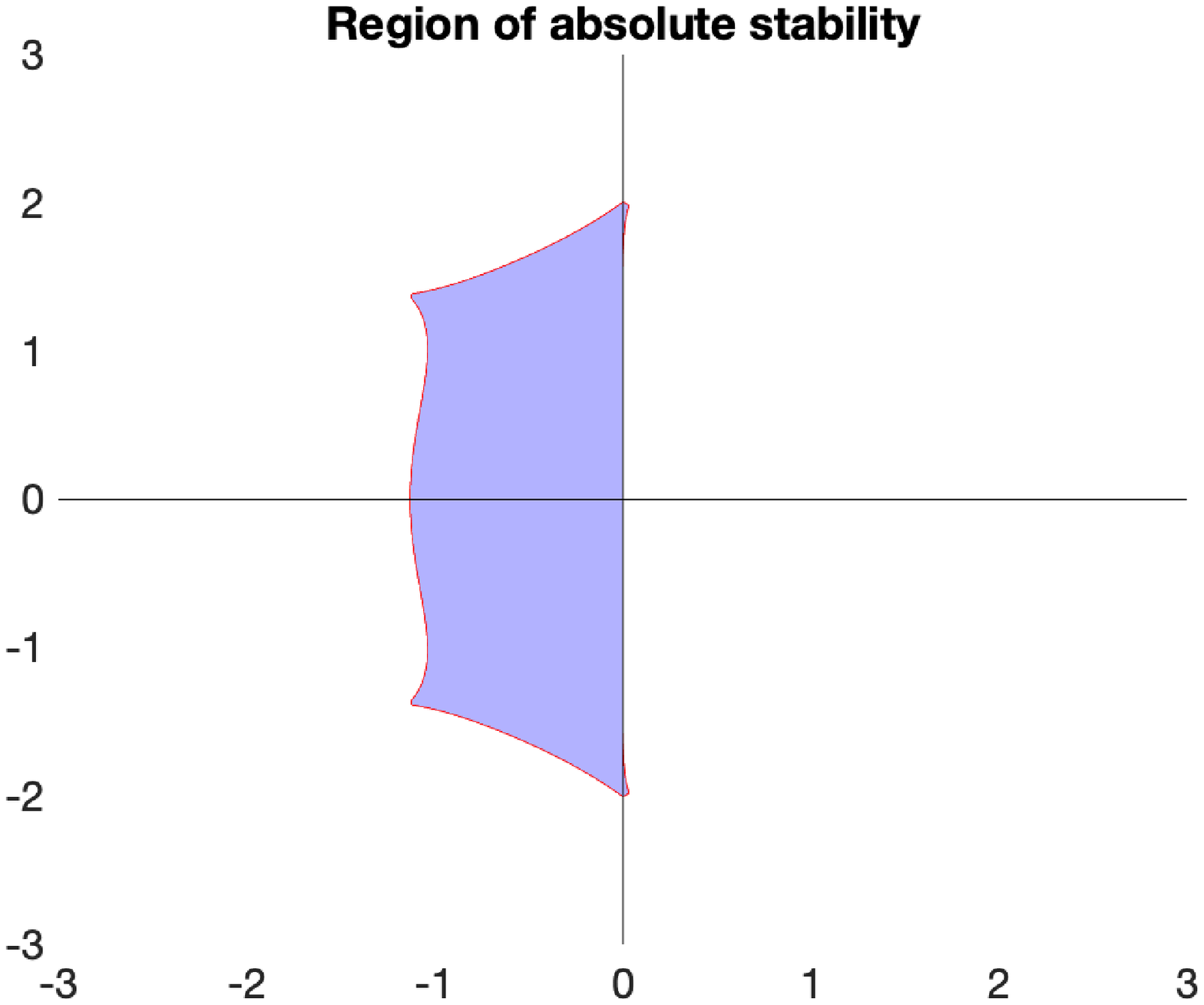} \\
    \includegraphics[width=.351\textwidth]{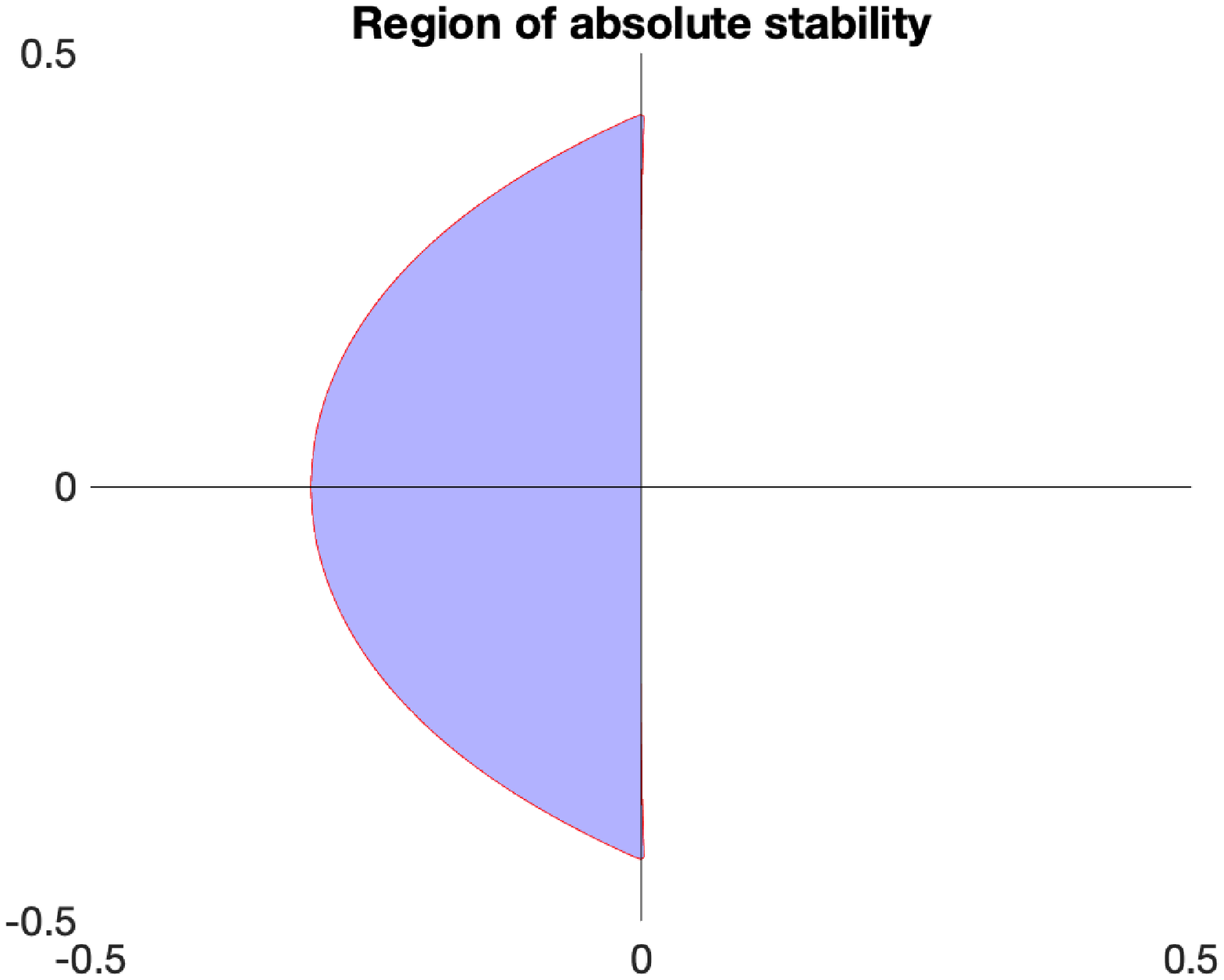} \hspace{-.2in}
  \includegraphics[width=.351\textwidth]{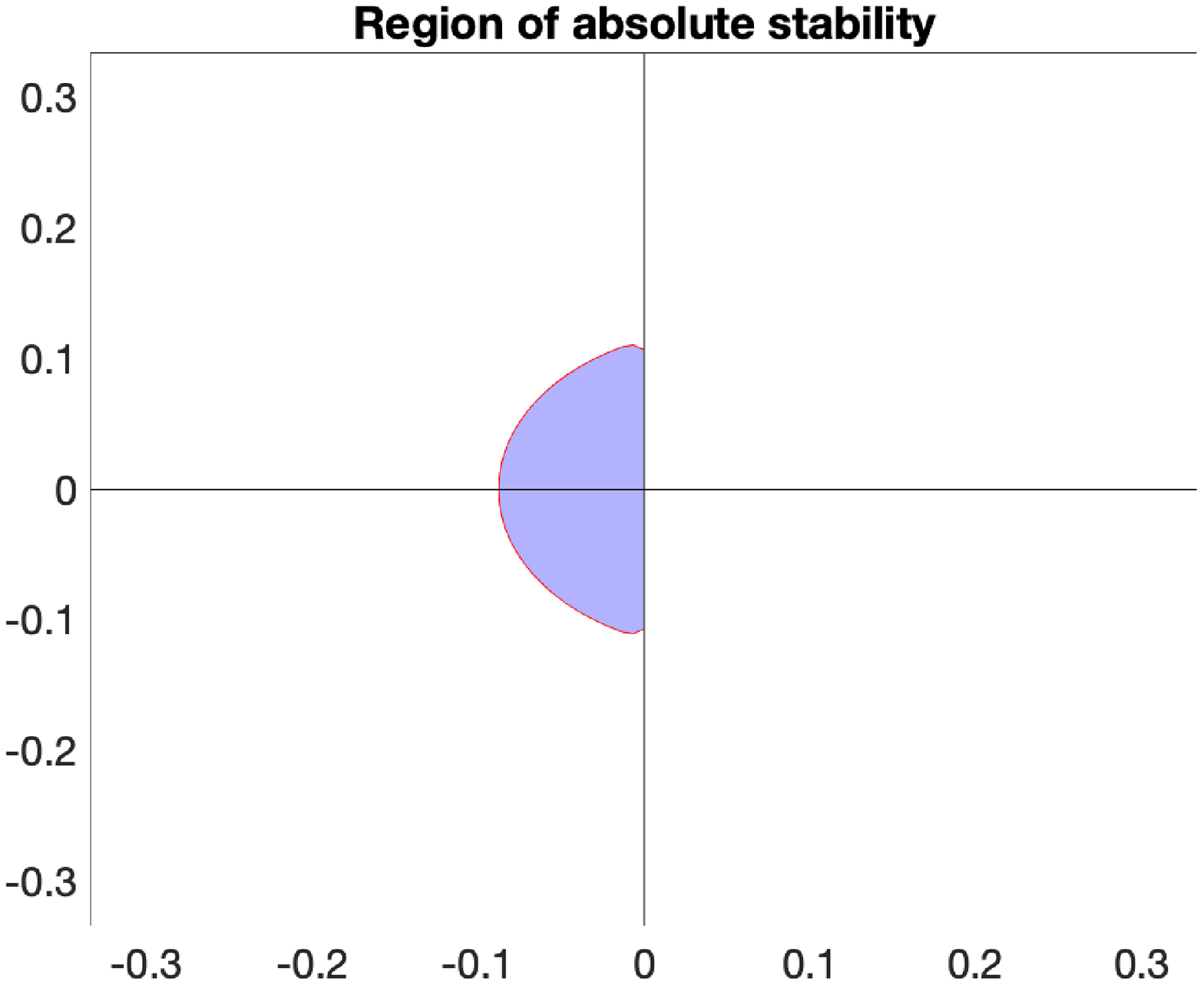}\hspace{-.2in}
   \includegraphics[width=.351\textwidth]{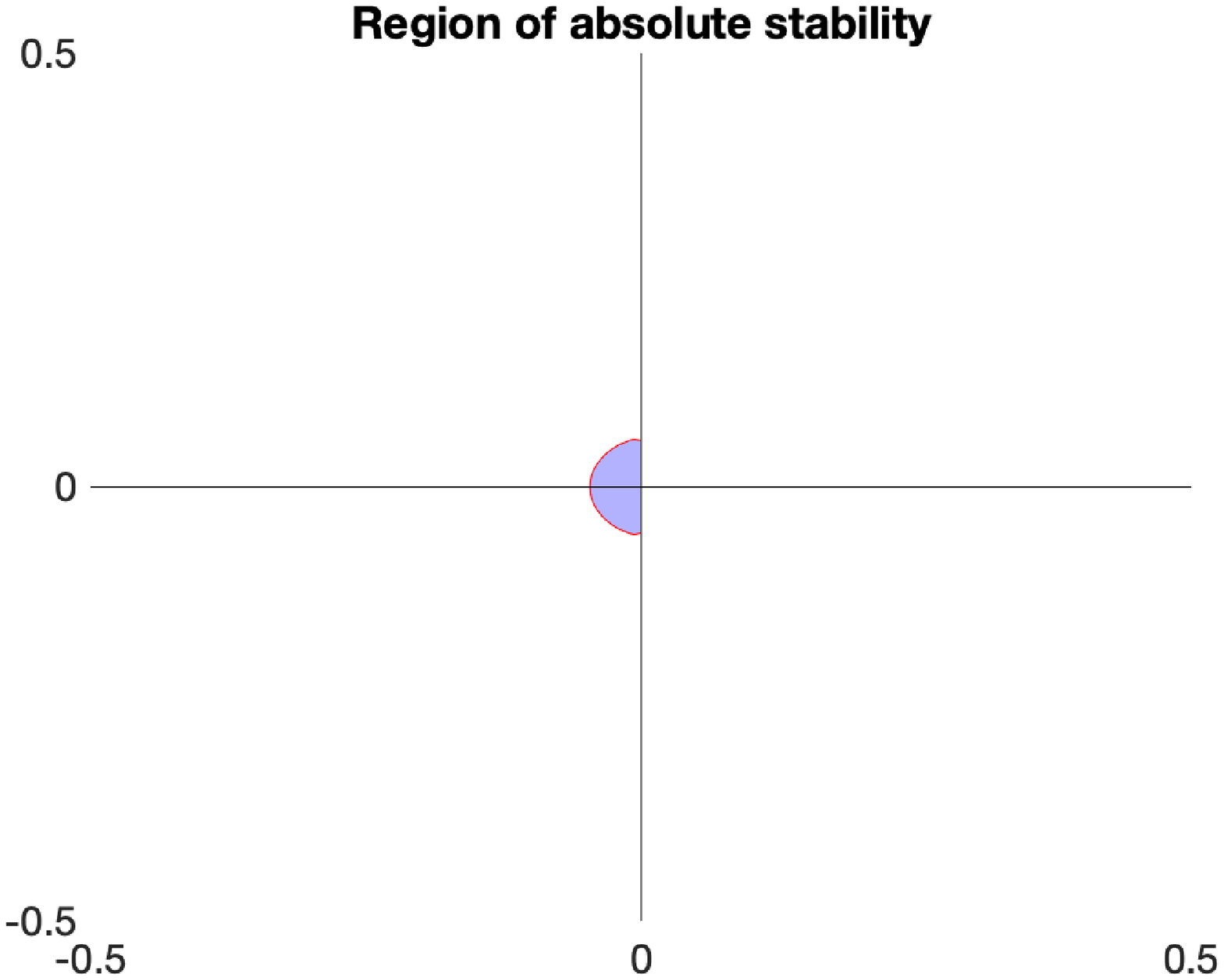} 
 \caption{\label{stability_regions} Stability regions os the eEIS+(2,4) (left),
 eEIS+(3,6) (middle), and eEIS+(5,7) (left). Bottom: stability regions of Adams Bashforth methods of the corresponding orders.}
\end{figure}

\noindent{\bf Explicit error inhibiting method eEIS+(3,6):}
This explicit three stage error inhibiting method is sixth
order if it is post processed, otherwise it gives fifth order accurate solutions.
The coefficients are:
\begin{eqnarray*}
 \mD &=& \left(\begin{array}{lll}
 d_1 & d_2 & d_3  \\
 d_1 & d_2 & d_3  \\
  d_1 & d_2 & d_3  \\
   \end{array}\right) 
   \; \; \; 
  \mbox{where} \; \; \; 
  \left(\begin{array}{l}
  d_1 \\ d_2 \\ d_3 \\ 
  \end{array} \right)=
  \left(\begin{array}{r}
0.844429704970785  \\  0.183161240819666   \\   -0.027590945790451 \\
  \end{array} \right)
\end{eqnarray*}
\vspace{-8pt}
\begin{eqnarray*}
\mA &=& \left(\begin{array}{ccc}
0.119782131013886   & 0.530075444729337 & 0.295068834365335  \\
0.034108245281186   &  0.972302193339061 & -2.090901330553469  \\
 -0.067206259640574 & 1.216836100819247 &  -0.661223528969050 \\
\end{array}\right) ,  
\end{eqnarray*}
and 
\vspace{-8pt}
\begin{eqnarray*}
 \mR &=& \left(\begin{array}{ccc}
 0 & 0 & 0 \\
2.464399360954857  & 0 & 0 \\
0.210685805002394  & 1.137368201889378  & 0 \\
\end{array}\right) . 
\end{eqnarray*}
The abscissas are $c_1= -0.891535334604278, \; \; 
c_2= -0.456552374616537, \; \; 
c_3=0.$ 
The truncation error vector is: \\
 $\tau_5= \left(     0.002851625181111, 
  -0.041196333074551,
  -0.186205087415322  \right)^T. $
The imaginary axis stability for this method is $(-0.5985,0.5985)$
To post-process this method we  take a linear combination of several numerical solutions:\\
$ \hat{v}^n = w_1 v^{n-1+c_1 } + w_2 v^{n-1+c_2} + w_3 v^{n-1} + w_4 v^{n+c_1} + w_5 v^{n+c_2} + w_6 v^{n}$\\
with weights:\\
$w_1=  -0.022895756757277 \; \;
w_2=   0.147460773700033 \; \;
  w_3=  -1.004504454589247$\\
  $w_4=   1.014066366026382 \; \;
    w_5=  -0.155617960794494 \; \;
    w_6=    1.021491032414602  .$

\smallskip

\noindent{\bf Explicit error inhibiting method eEIS(5,7):} 
This explicit three stage error inhibiting method is seventh
order if it is post processed, otherwise it gives sixth order accurate solutions.
The coefficients are given in the full $5 \times 5$  matrices $\mD, \mA$, and the strictly lower triangular 
matrix $\mR$. The coefficients in these matrices are:
\begin{eqnarray*}
 \mD &=& \left(\begin{array}{lllll}
 d_1 & d_2 & d_3 & d_4 & d_5 \\
 d_1 & d_2 & d_3 & d_4 & d_5 \\
 d_1 & d_2 & d_3 & d_4 & d_5 \\
 d_1 & d_2 & d_3 & d_4 & d_5 \\
  d_1 & d_2 & d_3 & d_4 & d_5 \\
   \end{array}\right) 
   \; \; \; 
  \mbox{where} \; \; \; 
  \left(\begin{array}{l}
  d_1 \\ d_2 \\ d_3 \\ d_4 \\ d_5 \\
  \end{array} \right)=
  \left(\begin{array}{r}
 -1.011623735666550\\ 
1.095449867712963 \\
1.789431260361622  \\
-0.872726291980225 \\  
-0.000531100427809 \\
  \end{array} \right)
\end{eqnarray*}   
\begin{eqnarray*}
& a_{11} =  0.542403428557849,  \; \; \;   & a_{12} = -0.760948514260222,  \; \; \;  a_{13} =  0.540150963081669,   \\
& a_{14} =  0.159072579950024,  \; \; \;   & a_{15} =  0.391433932478452, \\
& a_{21} =    0.156488609423175,  \; \; \;   & a_{22} = -0.242186890762633,  \; \; \;    a_{23} =    0.247855775765120,   \\
& a_{24} =       0.363064760009647,  \; \; \;   & a_{25} =   0.314695085548473,  
\end{eqnarray*}  
\begin{eqnarray*}
& a_{31} =    -0.052321607410313,  \; \; \;    & a_{32} = 0.097345632885763,  \; \; \;     a_{33} =    -0.221816006761698,  \\
& a_{34} =      0.900744500805372,  \; \; \;  &    a_{35} =   -0.013037891925596, \\
& a_{41} =    0.396379418407651,  \; \; \;  &  a_{42} = -0.498665400266501,  \; \; \;  a_{43} =     0.102234339427055,    \\
& a_{44} =    0.658422701253808,  \; \; \;  &  a_{45} =  -0.027557926231150, \\
& a_{51} =       1.449809317440111,  \; \; \;  &     a_{52} =-1.855043289819523,  \; \; \;    a_{53} =     0.795025316417296, \\ 
& a_{54} =      0.015237452869142,  \; \; \;   & a_{55} =   0.383077291565467. \\
& r_{21} =  .067750736449434,  \; \; \;  & r_{31} =-0.970866150021656,  \; \; \;   r_{32} = 1.411026181526863, \\ 
& r_{41} = 1.110541182884615,  \; \; \; & r_{42} =-0.861259710862469,  \; \; \;  r_{43} =  0.461581912124537, \\
& r_{51} =  0.142695702867824,  \; \; \; &  r_{52} =  0.803890471392162,  \; \; \;     r_{53} =-1.532866050532452, \\ 
& r_{54} =   1.507618973979455,  \; \; \; & \mbox{and $r_{ij} =0 $ whenever $j \geq i$}.
\end{eqnarray*}
    
The abscissas are 
$ c_1 = -0.837332796371710$, 
$ c_2 = -0.801777109746265$, \\
$ c_3 = -0.558370527080746$, 
$ c_4 = -0.367768669441936$,   
with $c_5 = 0.$ The truncation error vector is:
\[ \tau_{p+1}=\begin{bmatrix} -2.452136279362326 \times 10^{-3} \\
    -9.952624484663908 \times 10^{-4}  \\
    -6.583335089187866 \times 10^{-3} \\
    -1.186500759891287 \times 10^{-2} \\
    -6.616898102859160 \times 10^{-2} \\
    \end{bmatrix}
\] 
  The imaginary axis stability for this method is $(-2.0047, 2.0047)$.
    To post-process this method we  take a linear combination of several numerical solutions:\\
$ \hat{v}^n = w_1 v^{n-1+c_1 } + w_2 v^{n-1+c_2} + w_3 v^{n-1+c_3 } + w_4 v^{n-1+c_4} + w_5 v^{n} 
+ w_6 v^{n+c_1}+   w_7 v^{n+c_2}+ w_8 v^{n+c_3}+ w_9 v^{n+c_4}+ w_{10} v^{n} $\\
with weights:\\
$w_1=-0.108041130714896, \; \; 
   w_2=0.161475977012818, \; \;
 w_3= -0.205996099378955,$ \\
 $  w_4=0.317344948221968, \; \;
 w_5= -1.213968428247239, \; \;
 w_6=  6.439151511599838, $\\
$  w_7= -5.691821046332016, \; \;
   w_8=0.366796920786556, \; \; 
w_9=  -0.066491551558718, $
 $w_{10} =  1.001548898610644.$

\bigskip
Strong stability preserving methods \cite{SSPbook2011} are of interest for the time evolution of hyperbolic problems with shocks
or sharp gradients. These high order time-stepping methods preserve the nonlinear, non-inner-product
strong stability properties of the spatial discretization coupled with forward Euler time-stepping. The following two methods
show that it is possible to combine the EIS+ and SSP properties in a given method.

\noindent{\bf Explicit SSP error inhibiting method eSSP-EIS(3,4) }
This explicit three step error inhibiting method is fourth
order if it is post processed, otherwise it gives third order accurate solutions.
This method is strong stability preserving (SSP) with SSP coefficient $C= 0.7478$.
The coefficients of this method are: 

\begin{eqnarray*}
 \mD &=& \left(\begin{array}{lll}
 d_1 & d_2 & d_3  \\
 d_1 & d_2 & d_3  \\
  d_1 & d_2 & d_3  \\
   \end{array}\right) 
   \; \; \; 
  \mbox{where} \; \; \; 
  \left(\begin{array}{r}
  0.481236169483274 \\ 0 \\ 0.518763830516726\\
     \end{array}\right)  
     \end{eqnarray*}
\begin{eqnarray*}
    \mA &=& \left(\begin{array}{ccc}
    0 & 0 & 0.693711877859443\\
    0.081596114968722 & 0 & 0.333227135691426 \\
0.167078858485521 & 0 & 0.331269986340461 \\
       \end{array}\right)  \\
        \mR &=& \left(\begin{array}{ccc}
        0 & 0 & 0 \\
        0.642348436974698 & 0 & 0 \\
        0.254975180593489 & 0.530807045380761  & 0 \\
         \end{array}\right)  \\
\end{eqnarray*}
The abscissas are $ c_1 =   -0.590419192940789, \; \; 
c_2 =   -0.226959383165386, \; \; 
c_3 =  0.$ The truncation error vector is \\
$\ste_{p+1}= 10^{-2}  \times \left( -5.591881250375826,
    -5.080104811229902,
     5.187361482884723 \right).$
 To post-process this method we  take a linear combination of several numerical solutions:\\
$ \hat{v}^n = w_1 v^{n-1+c_1 } + w_2 v^{n-1+c_2} + w_3 v^{n-1 } +
 w_4 v^{n+c_1} + w_5 v^{n+c_2} + w_6 v^{n}$\\
with weights:\\
$w_1= -0.052886551536914, \; \; \;
w_2=    0.381993090397787, \; \; \;
w_3=    -0.580050146506483$ \\
$w_4=   0.439879549713232, \; \; \;
w_5=   -0.283052417950462, \; \; \;
w_6=    1.094116475882841.$

The SSP coefficient of this method compares favorably to the SSP coefficients of linear multistep methods:
to obtain a fourth order linear multistep method we need five steps and the SSP coefficient is small: $C=0.021$.
Even a linear multistep method with fifty steps has a smaller SSP coefficient $C=0.561$. 
However, a  comparison with Runge--Kutta methods is less favorable:
the low storage three-stage third order Shu-Osher SSP Runge--Kutta 
method has SSP coefficient $C= 1$
with the same number of function evaluations and lower storage. Scaled by the number of function evaluations we obtain an effective SSP
coefficient of $C/3 \approx 0.33$, whereas our current method has an effective SSP coefficient $C/3 \approx 0.25$. Our method is 
still competitive here because it is fourth order rather than third. However, if we compare to
 the five stage fourth order SSP Runge--Kutta  method which has SSP coefficient $C= 1.5$, ($C/5= 0.3$),
or to the low storage ten stage fourth order SSP Runge--Kutta  method has SSP coefficient $C= 6$ ($C/10= 0.6$),
our method is not as  efficient. 
The more correct comparison is to multi-step Runge--Kutta methods in \cite{msrk}: 
the three-stage, three-step fourth order method here has an effective SSP coefficient  $C/3 \approx 0.39$, which is higher than the 
eSSP-EIS(3,4).

\noindent{\bf Explicit SSP error inhibiting method eSSP-EIS(4,5) }
This four stage error inhibiting methods is fifth
order if it is post processed, otherwise it gives fourth order accurate solutions.
This method is strong stability preserving (SSP) with SSP coefficient $C=  0.643897$ (or an effective SSP 
coefficient $C/4\approx 0.16$).
This method is interesting because all SSP multistep methods require at least seven steps for fifth order and 
have very small SSP coefficients which make them inefficient.
There are no fifth order SSP Runge--Kutta methods \cite{SSPbook2011},
however we can  compare this method with the SSP multistep multistage methods in \cite{msrk}:
where the corresponding four step four stage method has effective SSP coefficient $C/4=0.38436$, which is more efficient.

The coefficients of the eSSP-EIS(4,5) are:
{\small
\begin{eqnarray*}
 \mD &=& \left(\begin{array}{llll}
 d_1 & d_2 & d_3 & d_4 \\
 d_1 & d_2 & d_3 & d_4 \\
  d_1 & d_2 & d_3 & d_4 \\
   d_1 & d_2 & d_3 & d_4 \\
   \end{array}\right) 
   \; \; \; 
  \mbox{where} \; \; \; 
  \left(\begin{array}{l}
  d_1 \\ d_2 \\ d_3 \\ d_4 \\
  \end{array} \right)
=
  \left(\begin{array}{r}
 0.391361993111787  \\
0.065690723540339    \\
0.209839489692975 \\
 0.333107793654898 \\
  \end{array} \right) 
  \end{eqnarray*}
  \begin{eqnarray*}
 \mA & = &   \left(\begin{array}{llll}
 0.111982379086567     & 0 & 0 &   0.517330861095791 \\
0.144956804626331     & 0 & 0 &  0.200688177229557 \\
 0.039506390225419    & 0.074215962133829 &  0.237072128025406 & 0.190419328868168 \\
 0.013111528886920     & 0.067038414113032  &  0.296412681422031 & 0.277723998040954 \\
   \end{array}\right) 
   \end{eqnarray*}
     \begin{eqnarray*}
   \mR & = &   \left(\begin{array}{llll}
   0 & 0 & 0 & 0 \\
0.602472175831079 & 0 & 0 & 0 \\
0.164197196121254  & 0.423264977696018 & 0 & 0 \\
0.054494380980164  &  0.140474767505132 & 0.515429866206022 & 0 \\
      \end{array}\right) \\
\end{eqnarray*}
}
The abscissas are $ c_1 =  -0.735372396971898 $,
$c_2 =-0.416568479467288$, \\
$c_3 =-0.236009654084161$, 
and $c_4 =0. $
The truncation error vector is 
\[
\tau_{p+1}=\begin{bmatrix} 
    -1.648864820077294 \\
    -4.617774532209270 \\
     0.7007842214544382  \\
     2.406415533885425 \\
    \end{bmatrix}  \times 10^{-2} \]
To post-process this method we  take a linear combination of several numerical solutions:\\
\[  \hat{v}^n = w_1 v^{n-1+c_1 } + w_2 v^{n-1+c_2} + w_3 v^{n-1+c_3} + w_4 v^{n-1} 
+ w_5 v^{n+c_1 } + w_6 v^{n+c_2} + w_7 v^{n+c_3} + w_8 v^{n}  \]
with weights:\\
$w_1= -0.052886551536914, \; \; \;
w_2=    0.381993090397787, \; \; \;
w_3=    -0.580050146506483$ \\
$w_4=   0.439879549713232, \; \; \;
w_5=   -0.283052417950462, \; \; \;
w_6=    1.094116475882841.$

\begin{rmk}
It is important to note that although these methods are SSP,  the post-processor is not guaranteed to preserve the strong stability 
properties of the solution. It is entirely possible that the $O(\dt^{p+1})$ accurate solution at the final time will satisfy the strong
stability property of interest and the post-processed solution will not.  In practice, this is not a major concern for the following reasons:
\begin{enumerate}
\item The post-processor is applied only once at the end of the simulation, and only impacts the solution at the final time. 
Typically, the strong stability  properties  we require are only important for the stability of the time evolution, 
but not necessarily needed at the final time. In such cases, the nonlinear strong stability properties that we are concerned about 
(e.g. positivity or total-variation boundedness)
are critical to the time-evolution itself. If the time-stepping method is not SSP, the code may crash due to non-physical negative values
in some of the quantities (such as pressure or density), or the method may become unstable due to non-physical oscillations
which grow and destroy the solution. However, a post-processor that is applied only once at the end of the time-simulation will not impact
the nonlinear stability of the time-evolution and so is not of concern.
Furthermore, if the final post-processing step results in a solution that has undesired characteristics, this higher order solution can be ignored.
\item The post-processing step is a simple projection that, as we saw above, changes the solution only on the order of $\dt^{p+1}$.
Thus, any violation of the strong stability properties are only at the level of $O(\dt^{p+1})$, which is very small and thus will not have 
major impact on the strong stability properties of the solution.
\end{enumerate}
These arguments are validated in our numerical tests, where we see that the post-processing does not have an adverse impact on the simulations
and leads to very small differences in the strong stability properties of interest between the final time solution and the post-processed solution.
 \end{rmk}

\subsection{Implicit Methods}
In this section we present several implicit error inhibiting methods that can be post-processed. 
For each method we present the coefficients $ \mD, \mA, \mR$, as well as values of the 
abscissas and the truncation error vector that must be used to construct the 
postprocessor. We denote an implicit $s$-step method that satisfies the conditions \eqref{EIS2conditions} 
and can be postprocessed to order $p$ as iEIS+($s$,$p$).  All the methods we present are 
A-stable, so we do not show their stability regions here.

\noindent{\bf A-stable implicit method  iEIS+(2,3).}
This A-stable implicit two stage error inhibiting method is third order if it is post processed, 
otherwise it gives second order accurate solutions.  The coefficients are given in:
  \vspace{-4pt}
\begin{eqnarray*}
& \mD = \left(\begin{array}{ll}
      2      &       -1      \\
       2     &        -1    \\
   \end{array}\right) ,     \; \;   
  & \mA = \frac{1}{12} \left(\begin{array}{ll}
     13 &          -14   \\
     16       &   -24   \\  
   \end{array}\right) ,    \; \; 
  \mR =   \frac{1}{12}  \left(\begin{array}{ll}
     19  &           0   \\    
      24 &      8 \\     
          \end{array}\right) .
  \end{eqnarray*}         
Here the abscissas are  $ c_1=-\frac{1}{2}, c_2=0, $    and  the truncation error vector is
$\tau = \left( \frac{3}{8}   , \frac{3}{4}        \right)^T.$
To post-process this method we  take a linear combination of several numerical solutions:\\
\[  \hat{v}^n = \frac{1}{2} v^{n-\frac{3}{2}  } - \frac{3}{2}  v^{n-1} + \frac{3}{2} v^{n-\frac{1}{2}} + \frac{1}{2} v^{n} . \]
       
       The cost of the implicit solve is often non-trivial, and the computation can be speeded up if 
       the method admits an efficient parallel implementation. For this reason, we added the requirement 
       in our optimization code that  $\mR$ is a diagonal matrix. All the following methods are efficient when 
       implemented in parallel.
       
\smallskip 

\noindent{\bf Parallel-efficient A-stable implicit method  iEIS+(2,3).}
This two stage error inhibiting method is third
order if it is post processed, otherwise it gives second order accurate solutions.
This method is A-stable implicit with diagonal $\mR$ allowing the implicit solves to be performed concurrently.  
The coefficients are given in:

\begin{eqnarray*}
& \mD = \frac{1}{15} \left(\begin{array}{ll}
16    &      -15    \\
16  &      -15    \\
 \end{array}\right) ,  \; \; 
&  \mA = \frac{1}{480}
\left(\begin{array}{ll}
 75     &    106   \\
-1440         &   736   \\
   \end{array}\right) ,      \; \; 
 \mR =\frac{1}{32}  \left(\begin{array}{ll}
21    &      0        \\
       0          &    96        \\
          \end{array}\right) .
 \end{eqnarray*}
Here the abscissas are
 $ c_1=- \frac{1}{2}, c_2=0, $
   and  the truncation error vector is
$\tau = \frac{1}{120}, \left( 31 , 496 \right)^T.$ 
To post-process this method we  take a linear combination of several numerical solutions:\\
\[  \hat{v}^n = \frac{4}{15} v^{n-\frac{3}{2}   } - \frac{4}{5}  v^{n-1} + \frac{4}{5} v^{n-\frac{1}{2}  } + \frac{11}{15} v^{n} .\]

\smallskip

\noindent{\bf Parallel-efficient A-stable  implicit method iEIS+(3,4)} 
This three stage error inhibiting method is fourth
order if it is post processed, otherwise it gives third order accurate solutions.
This method is A-stable implicit with diagonal R allowing the implicit solves to be performed concurrently.  
The coefficients are given in:
  \vspace{-8pt}
 {\small
 \begin{eqnarray*}
 \mD &=& \left(\begin{array}{llll}
 d_1 & d_2 & d_3  \\
 d_1 & d_2 & d_3  \\
  d_1 & d_2 & d_3  \\
   \end{array}\right) 
   \; \; \; 
  \mbox{where} \; \; \; 
  \left(\begin{array}{l}
1.100594730800523  \\ -0.335370831614021   \\ 0.234776100813498 \\
  \end{array} \right)
  \end{eqnarray*} 
  \vspace{-14pt}
\begin{eqnarray*}
\mA &=& \left(\begin{array}{lll}
   0.806950212712456 &  -0.386181733528596  & -0.182046279153154 \\
   2.687898652721551  & -1.944296251569286  & -1.165162710461159 \\ 
   1.052813949541399  & -0.265689012035030  & -0.052553462549502 \\
 \end{array}\right) 
   \end{eqnarray*} 
\begin{eqnarray*}
\mR &=& \left(\begin{array}{lll}
   0.716550676631637          &         0        &           0 \\
                   0   & 1.710166519304569        &           0 \\
                   0     &     0   & 0.887368068372141 \\
 \end{array} \right) 
 \end{eqnarray*}
 }
 Here the abscissas are
 $ c_1=-\frac{2}{3},  c_2=-\frac{1}{3}, c_3=0, $
   and  the truncation error vector is
\[\tau = \left( 0.278446186799822, 1.535336949555884,  0.887870711092943 \right)^T.\]   
    To post-process this method we  take a linear combination of several numerical solutions:\\
$ \hat{v}^n = w_1 v^{n-1+c_1 } + w_2 v^{n-1+c_2} + w_3 v^{n-1 } +
 w_4 v^{n+c_1} + w_5 v^{n+c_2} + w_6 v^{n}$\\
with weights:\\
$w_1 =  -0.133267324186142, \; \; \;
w_2 =     0.666336620930711, \; \; \;
w_3 =    -1.332673241861424,$
$w_4 =     1.332673241861424, \; \; \;
w_5 =    -0.666336620930712, \; \; \;
w_6 =     1.133267324186142.$

\smallskip

\noindent{\bf Parallel-efficient A-stable implicit method iEIS+(4,5). }
This four stage error inhibiting method is fifth
order if it is post processed, otherwise it gives fourth order accurate solutions.
This method is A-stable implicit with diagonal R allowing the implicit solves to be performed concurrently.  
The coefficients are given in:
{\small
\begin{eqnarray*}
 \mD &=& \left(\begin{array}{llll}
 d_1 & d_2 & d_3 & d_4 \\
 d_1 & d_2 & d_3 & d_4 \\
  d_1 & d_2 & d_3 & d_4 \\
   d_1 & d_2 & d_3 & d_4 \\
   \end{array}\right) 
   \; \; \; 
  \mbox{where} \; \; \; 
  \left(\begin{array}{l}
  d_1 \\ d_2 \\ d_3 \\ d_4 \\
  \end{array} \right)
=
  \left(\begin{array}{r}
   -2.189053680903935 \\
    3.606949225806165 \\
 -0.710842571233197 \\ 
0.292947026330966 \\
  \end{array} \right) \\
  \mR &=&   \left( \begin{array}{cccc}
    r_1 & 0 & 0& 0 \\
    0 & r_2 & 0 & 0 \\
    0 & 0 & r_3 & 0 \\
    0 & 0 & 0 & r_4 \\
      \end{array} \right)  
        \; \; \;  \mbox{where} \; \; \; 
       \left(\begin{array}{l}
  r_1 \\ r_2 \\ r_3 \\ r_4 \\
  \end{array} \right)
=
  \left(\begin{array}{r}
0.243205109444297 \\
0.428641943283907 \\
1.223508778356526 \\
0.861606621761651 \\
    \end{array} \right)
   \\
        \mbox{and:} \\
\mA &=& \left( \begin{array}{rrrr}
 0.542633235622690  &  0.572906890966515 & -0.147775065138658  & 0.108270009767368 \\
-0.935354930827541  & 1.187517922840311 &  0.040246733851822  & -0.237077959731666 \\
-3.856502347754360  &  5.000000000000000  &  3.366967278814666  & -5.000000000000000 \\
-3.605680346039871 & 4.951687114045852  & 1.612027197556519 &  -2.835666877907317 \\
  \end{array} \right)   
      \end{eqnarray*} 
   }
   Here the abscissas are $ c_1=-\frac{3}{4},  c_2=-\frac{1}{2}, c_3=-\frac{1}{4}, c_4=0, $
  and  the truncation error vector is
\[ \tau_p=\begin{bmatrix}   0.044949370534240\\
   0.165996341680758\\
   1.268926100495425\\
   1.371111036428543 \end{bmatrix} .\]
To post-process this method we  take a linear combination of several numerical solutions:\\
\[  \hat{v}^n = w_1 v^{n-\frac{7}{4} } + w_2 v^{n-\frac{3}{2}} + w_3 v^{n-\frac{5}{4}} + w_4 v^{n-1} 
+ w_5 v^{n-\frac{3}{4}} + w_6 v^{n-\frac{1}{2}} + w_7 v^{n-\frac{1}{4}} + w_8 v^{n}  \]
with weights:\\
$w_1 =   0.081324340500950, \; \; \;
w_2 =    -0.569270383506653, \; \; \;
w_3 =     1.707811150519959, $\\
$w_4 =    -2.846351917533271, \; \; \;
w_5 =     2.846351917533285, \; \; \;
w_6 =    -1.707811150519988, $ \\
$w_7 =     0.569270383506672, \; \; \;
w_8 =     0.918675659499045.$

\bigskip

\section{Numerical Results\label{sec:test}}

In this section we test the methods presented in Section 4 on a selection of numerical test cases. Most of the
numerical studies are designed to show that the methods achieve the desired convergence rates
on nonlinear scalar and systems of ODEs, as well as systems of ODEs resulting from semi-discretizations of PDEs. 
We also explore the behavior of the SSP methods in terms of preserving the total variation diminishing properties of 
the spatial discretizations, and the issue of order reduction that occurs in implicit methods.

\subsection{Comparison of explicit schemes}

In Section 4.1 we presented several explicit EIS schemes that can be post-processed to attain higher order.
Here we demonstrate that these methods attain the design order of convergence on several standard test 
cases. From the point of view of practical implementation, we are interested in the computational cost needed
to attain a certain level of accuracy. To compute this, we look at the number of time-steps needed 
with and without post-processing to attain an error of a given size.
Since the cost of post-processing is a linear combination of some solutions, we consider this cost to be at most
that of one   function evaluation.
We comment on the the relative cost to achieve a certain accuracy, and show that using 
the  post-processor enables a far more efficient  computation.

\smallskip

\noindent{\bf Nonlinear scalar ODE:}
We compare the performance of several two-step schemes on the nonlinear ODE
\[ y'= - y^2\]
with initial condition $y(0)=2$. The methods we consider here are: 
\begin{itemize}
\item A non-EIS two step second order method presented by Butcher in \cite{{butcher1993a}}
{
\[ V^{n+1}  =   
\frac{1}{4} 
\left( \begin{array}{ccc}
-3& \;\; & 7 \\
-3 &  & 7
\end{array} \right ) V^n + 
\frac{\dt}{8}\left( \begin{array}{ccc}
-3  & \;&-3 \\
-7 & & 9
\end{array} \right )  F(V^n)
 \]}
abscissas are $c_1=1,c_2=2$. 
(Note that the abscissas are different in John Butcher's formulation).
\item An eEIS(2,3) method presented in \cite{EISpaper1}
{
\[ 
V^{n+1}  \,=  \, \frac{1}{6}\left( \begin{array}{ccc}
7& \;\; & -1 \\
7 &  & -1
\end{array} \right ) V^n + 
\frac{\dt}{24}\left( \begin{array}{ccc}
1  & & 125 \\
-17 & \;& 55 \\
\end{array} \right ) F(V^n)
 \]}
 abscissas are $c_1=-\frac{1}{2},c_2=0$.
\item The eEIS+(2,4) method presented in Section 4.1.
\end{itemize}

\begin{figure}[hbt]
    \centering
           \includegraphics[width=.85\textwidth]{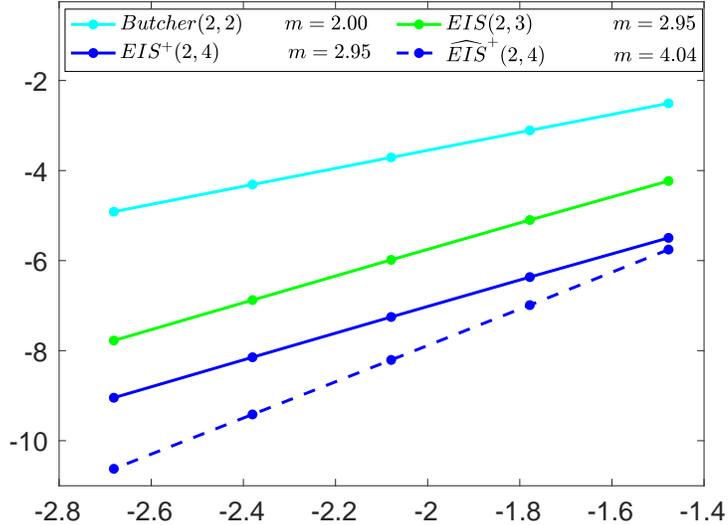}
           \caption{Convergence of several two-step schemes on a nonlinear scalar test problem.
           On the x-axis is  $log_{10}$ of time-step and on the y-axis is $log_{10}$ of the errors.
           The  non-EIS method by Butcher (in cyan solid) shows second order convergence , while the EIS method (in green solid) shows third order.
The eEIS+(2,4) method (blue solid) is third order as well, but with a smaller error constant. 
Finally, when the results of the eEIS+(2,4) method for the final time $T_f=1$ are post-processed (blue dashed),  fourth order convergence is obtained.
\label{comparison1} }
\end{figure}

Figure \ref{comparison1} shows the convergence history of each of these methods. 
The  non-EIS method by Butcher (in cyan) shows clear second order, 
while the EIS method (in green) shows third order.
The eEIS+(2,4) method (blue, denoted $EIS$+ in legend) is third order as well, 
but with a smaller error constant. 
Finally, when the results of the  eEIS+(2,4) method for the final time $T_f=1.0$ 
are postprocessed ( dashed blue line),  
fourth order convergence is obtained (denoted $\widehat{EIS}_1^+$ and $ \widehat{EIS}_2^+$  in legend, respectively).

\bigskip

\noindent{\bf Non-stiff Van der Pol oscillator:}
The nonlinear system of ODEs is given by
\[ \left( \begin{array}{l} y_1 \\ y_2 \end{array} \right)'
= \left( \begin{array}{c} y_2 \\ (1-y_1^2) y_2 - y_1 \end{array} \right)
\]
with initial condition $\vy(0) = (2,0)^T$. We use the explicit methods eEIS+(2,4),
eEIS+(3,6), eEIS+(5,7) to evolve this problem to the final time $T_f=2.0$ and postprocess the
solution at the final time as described in Section \ref{sec:postproc}. In Figure \ref{VDPexplicit} we show the errors
for different values of $\dt$ for $y_1$ on the left and $y_2$ on the right. 
The errors before post-processing are in solid lines, after post-processing  are dashed lines. 
The slopes of these lines
are computed using MATLAB's {\tt polyfit} function and are shown in the legend.
This example verifies that  numerical solutions from the eEIS+ methods 
attain the expected orders of convergence with and without post-processing.

\smallskip

We also see the impact of post-processing here: to attain an accuracy of $10^{-6}$ the eEIS+(2,4) 
method with no post-processing requires a stepsize $\dt \approx 0.0137$, or approximately 145 time-steps
which constitute 290 function evaluations.
In comparison,  the eEIS+(2,4)  method with post-processing will attain an accuracy of $10^{-6}$
with a stepsize $\dt \approx 0.0317$, or approximately 63 time-steps which means 126 function evaluations.
The additional cost of post-processing is less than one function evaluation, so we obtain a speedup of a 
factor of about 2.28. 
For the eEIS+(3,6) method to attain a level of accuracy of $10^{-9}$, we require 158 time-steps 
without post-processing and 91 time-steps with the post-processor.
If we count the cost of the post-processing as one  function evaluation,
using the  post-processor gives us a speedup of a factor of 1.72.
For the eEIS+(5,7) method to attain a level of accuracy of $10^{-11}$, we require 132 time-steps 
without post-processing and 75 time-steps with post-processor. 
Once again, if we count the cost of the post-processing as one  function evaluation,
using  post-processor gives us a speedup of a factor of 1.75.

\begin{figure}
    \centering
          \includegraphics[width=.52\textwidth]{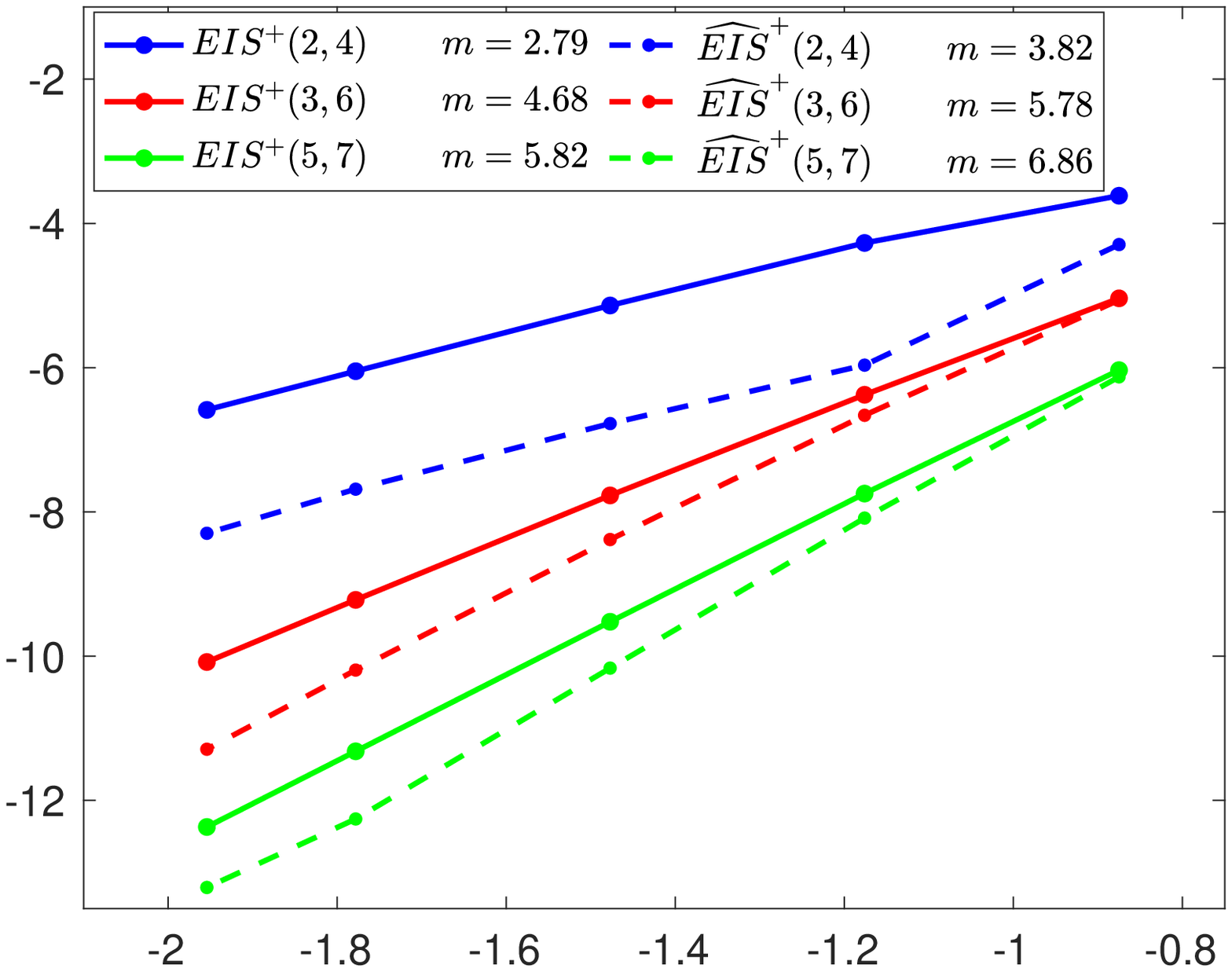} \hspace{-.325in}
            \includegraphics[width=.52\textwidth]{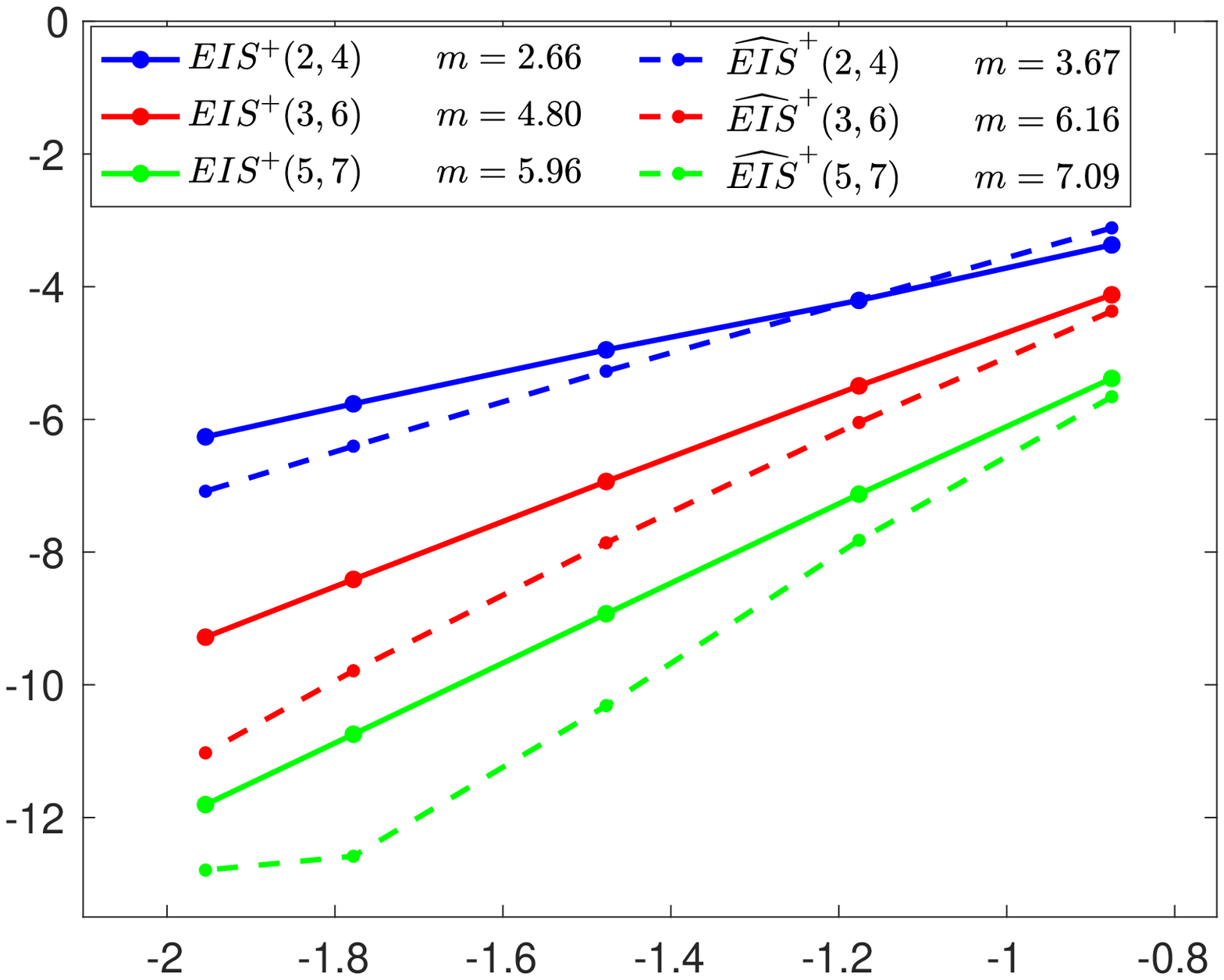}
           \caption{Convergence of the postprocessed solution from evolving the Van der Pol 
           system to time $T_f=2.0$ using eEIS+(2,4) (blue), eEIS+(3,6) (red), and eEIS+(5,7) (green)
            methods for the non-stiff Van der Pol system. 
The errors before post-processing are in solid lines (slopes given by $m$), 
after post-processing  are dashed lines. 
Left: the $log_{10}$ errors in the first element, $y_1$, for various time-steps.
Right: the same for the second element $y_2$.
\label{VDPexplicit} }
\end{figure}

\bigskip

\noindent{\bf Advection-diffusion problem:}
We solve the advection--diffusion problem
\[ u_t +a u_x = b u_{xx} \]
with periodic boundary conditions $u(0,t) = u(2 \pi, t)$ and initial condition
\[ u(x,0) = \sin(5 x).\] Here we use $a=1.0$ and $b=0.1$.
We discretize the problem in space with $N=41$ points
using a Fourier spectral method to obtain the ODE
\[ y' = \left( -D_x   + \frac{1}{10} D_x^2 \right) y,\]
where $D_x$ is  the first order Fourier differentiation matrix and $D_x^2$ is the second order Fourier differentiation matrix. Due to the periodicity of this problem, the spatial differentiation is exact and so the errors arise only from the time evolution of this ODE.

We evolve this problem forward to final time $T_f=1.0$ using the eEIS+(s,p) methods presented in Section 4.1
with $\dt = T_f / M$,
and postprocess the solution at the final time as described in Section \ref{sec:postproc}. 
In Table \ref{explicitAdvDiff} we show the errors for  different values of $M = 1/\dt$ before 
and after postprocessing. We observe that the errors are of the predicted EIS order
(which is one order higher than predicted by a truncation error analysis)
before post-processing and gain an order after postprocessing, as expected.

\bigskip

Notice that post-processing gives us much better accuracy for a very small price:
for example,  using the eEIS+(2,4) method  without post-processing with 300 time-steps  
($\dt \approx 0.0033$)  we obtain an  error of $2.16 \times 10^{-7} $.
 but if we use 
the eEIS+(2,4) method  with post-processoring, we need only 150 time-steps 
($\dt \approx 0.0066$) to obtain an even smaller error of $1.96\times 10^{-7}$. 
Clearly, using the post-processor speeds up our time to a solution of desired accuracy by 
slightly more than a factor of two.  
Similarly, using  the eEIS+(3,6) we can obtain a  final-time solution of accuracy $6.90\times 10^{-12}$  
without post-processing using 300 time-steps, while with post-processing 
we require only 200 time-steps to obtain a final-time solution of accuracy $ 7.34\times 10^{-12}$.
Finally, using  the eEIS+(5,7) we can obtain a 
final-time solution of accuracy $2.22\times 10^{-10}$  without post-processing using 55 time-steps, while with post-processing 
we require only 45 time-steps to obtain a better final-time solution of accuracy $ 1.43\times 10^{-10}$.
Using these methods with time-stepping allows for a more efficient production of an accurate solution.

\begin{table} \label{explicitAdvDiff}
\begin{center}
\begin{tabular}{|l|l|ll|ll|} \hline
&&  \multicolumn{2}{c|}{at final time}& \multicolumn{2}{c|}{post-processessor} \\ \hline
method & M & error &   order & error & order \\ \hline
eEIS+(2,4)
&100  &  6.52$\times 10^{-6} $   &    --   & 1.01$\times 10^{-6}$ &    --   \\
&150  &  1.83$\times 10^{-6} $   & 3.13    & 1.96$\times 10^{-7}$ &  4.04   \\
&200  &  7.52$\times 10^{-7} $   & 3.09    & 6.16$\times 10^{-8}$ & 4.03 \\
&250  &  3.78$\times 10^{-7} $   & 3.07    & 2.50$\times 10^{-8}$ & 4.02\\
&300  &  2.16$\times 10^{-7} $   & 3.06    & 1.20$\times 10^{-8}$ & 4.02 \\ \hline
eEIS+(3,6)
&100  &  1.94$\times 10^{-9} $ &    --  &  4.90$\times 10^{-10}$ &    --           \\
&150  &  2.37$\times 10^{-10}$ &  5.18  &  4.19$\times 10^{-11}$ &  6.06   \\
&200  &  5.44$\times 10^{-11}$ &  5.12  &  7.34$\times 10^{-12}$ &  6.05   \\
&250  &  1.74$\times 10^{-11}$ &  5.09  &  1.91$\times 10^{-12}$ &  6.02   \\
&300  &  6.90$\times 10^{-12}$ &  5.08  &  6.52$\times 10^{-13}$ &  5.90  \\ \hline
eEIS+(5,7)
&35  &  3.34$\times 10^{-9} $  &   --   &  8.27$\times 10^{-10}$   &  --        \\
&40  &  1.50$\times 10^{-9} $  &  6.00  &  3.25$\times 10^{-10}$ & 6.97    \\
&45  &  7.41$\times 10^{-10}$  &  5.99  &  1.43$\times 10^{-10}$ & 6.98  \\
&50  &  3.94$\times 10^{-10}$  &  5.99  &  6.86$\times 10^{-11}$ & 6.98  \\
&55  &  2.22$\times 10^{-10}$  &  5.99  &  3.52$\times 10^{-11}$ & 6.99    \\ \hline
    \end{tabular}
    \end{center}
   \caption{Convergence of the solution from evolving the advection-diffusion equation
problem to time $T_f=1$ using different  eEIS methods with and without post-processing.
   }
    \end{table}
    
 \bigskip
 
 Next, we study the SSP properties of the eSSP-EIS+ schemes  presented in Section 4.1.
 To do so, we look at a problem where the spatial discretization is total variation diminishing
 when coupled with forward Euler time-stepping, and examine the maximal rise in total variation
 when this problem is evolved forward with an eSSP-EIS+ scheme.

\noindent{\bf SSP study:}
As two of our explicit methods are strong stability preserving, we demonstrate their ability to preserve
the total variation diminishing property of a first-order upwind spatial difference applied to Burgers' equation
with step function initial conditions:
  \begin{align}
u_t + \left( \frac{1}{2} u^2 \right)_x & = 0 \hspace{.75in}
    u(0,x)  =
\begin{cases}
1, & \text{if } 0 \leq x \leq 1/2 \\
0, & \text{if } x>1/2 \nonumber
\end{cases}
\end{align}
on the domain $[0,1]$ with periodic boundary conditions.
We used a first-order upwind difference to semi-discretize, with $200$ spatial points, the 
nonlinear spatial  terms $N(u) \approx -  \left( \frac{1}{2} u^2 \right)_x$.

We evolve the solution 10 time-steps using the SSP methods eSSP-EIS(3,4) and eSSP-EIS(4,5) 
for different values of $\dt$. 
At each time-level $y^n$ we compute the total variation of the solution at time $u^n$,  
\[ \| u^n \|_{TV} = \sum_{j} \left| u^n_{j+1} - u^n_j \right| . \]
The maximal rise in total variation (solid line with  circle markers) 
is graphed against $\lambda = \frac{\dt}{\dx}$ in Figure \ref{SSPtest}.
We then post-process the solution at the final time for all values of $\dt$ before the total variation begins to rise,
and compute the difference between the total variation of the solution at the 
final time and the postprocessed solution
\[  \| u^n \|_{TV}  -  \| \widehat{u}^n \|_{TV}  \]

\begin{figure}[t]
    \centering
          \includegraphics[width=.51\textwidth]{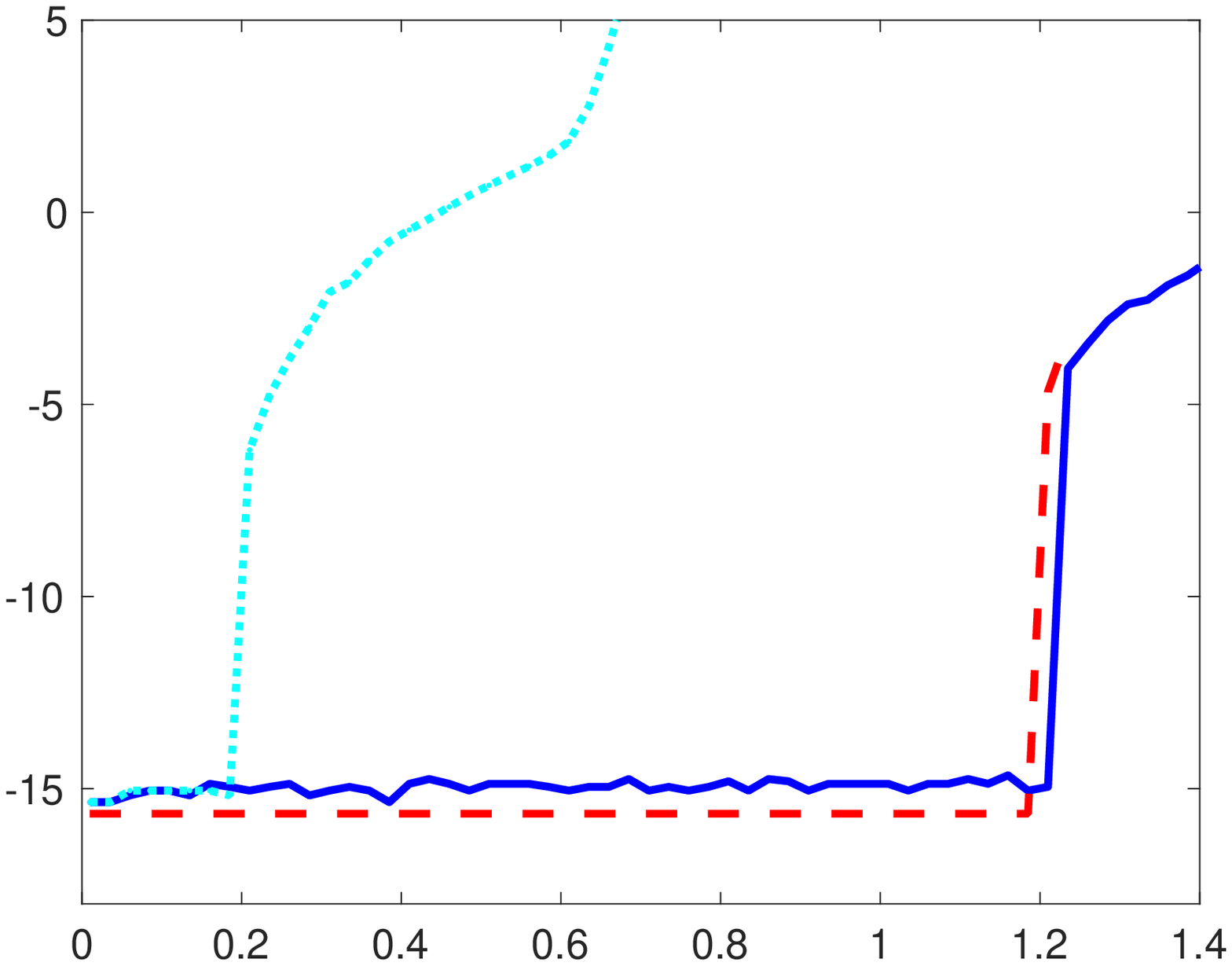} \hspace{-0.2in}
            \includegraphics[width=.51\textwidth]{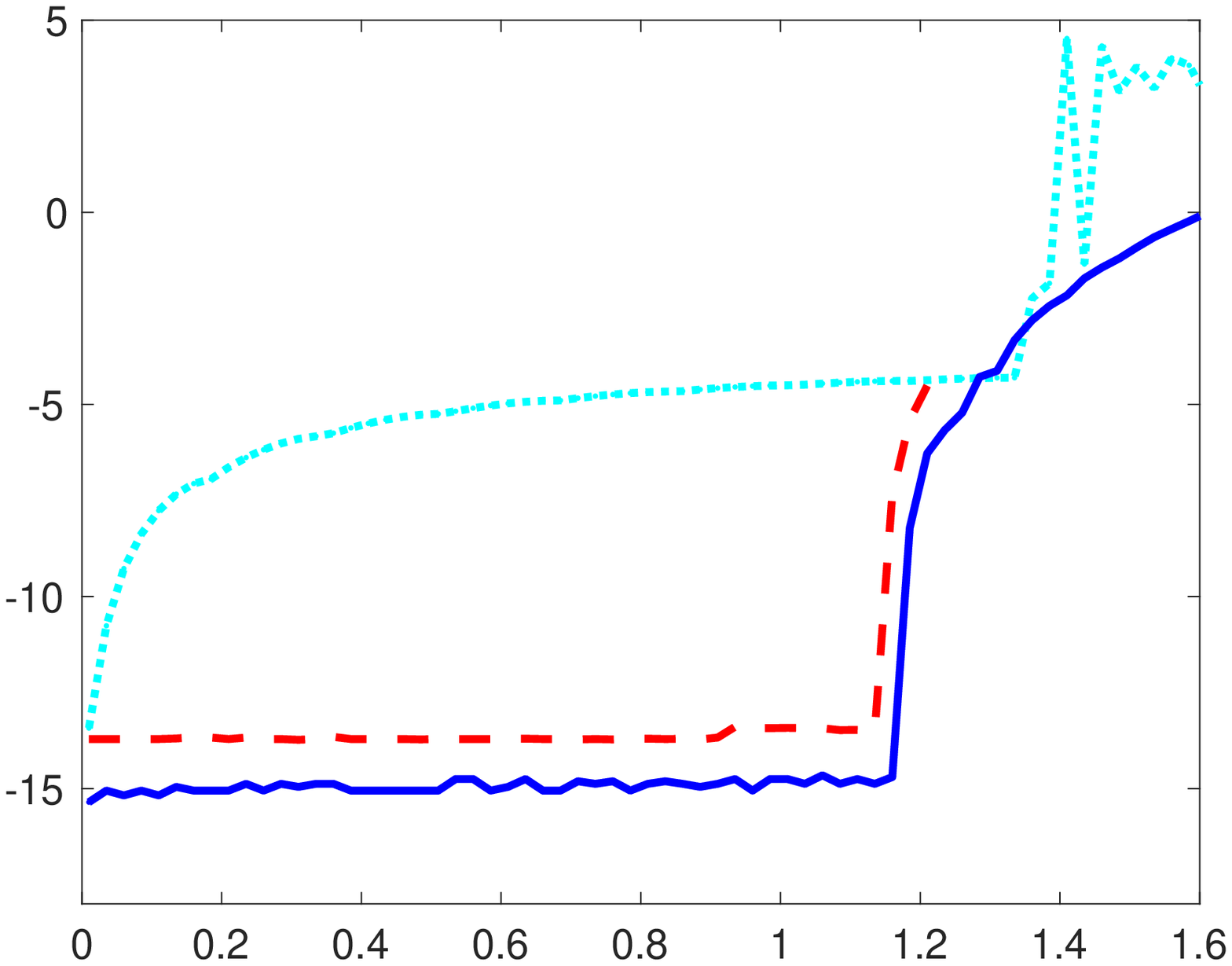}
           \caption{SSP study: the eSSP-EIS methods were used to evolve forward the solution in time
           of Burgers' equation with a step function initial conditions.
           On the y-axis is $log_{10}$ of the maximal rise in total variation, and 
           on the x-axis   the CFL number $\lambda$. The blue  solid line is the maximal rise in total variation from the 
           eSSP-EIS method without post-processing. The cyan dotted line is the maximal rise in total variation from the non-SSP method.
           The red dashed line is the difference between the total variation of the un-processed solution and that of the 
           post-processed solution. Left: eSSP-EIS+(3,4) compared to the eEIS+(2,4) method. 
           Right: eSSP-EIS+(4,5) compared to the implicit eEIS+(4,5) method.
            \label{SSPtest}  }
\end{figure}

In the left graph of Figure \ref{SSPtest} we see that the eSSP-EIS(3,4) preserves the
 total variation diminishing property up to $\lambda \approx 1.2$
(larger than the predicted value of $\lambda \leq 0.75$).
We compare the maximal rise in total variation from the eSSP-EIS(3,4) method (blue solid line) to the
maximal rise in total variation from the eEIS(2,4) method (cyan dashed line), which is not SSP, 
and indeed we see that the total variation is comparatively large for even small values of $\dt$.
The  maximal difference between the total variation of the
solution and the post-processed solution (red dashed line) also remains small $(\approx 10^{-14})$.
 
On the right graph of Figure \ref{SSPtest} we see that the eSSP-EIS(4,5) preserves the
 total variation diminishing property up to $\lambda \approx 1.16$ (larger than the predicted value of $\lambda \leq 0.63$).
We compare the eSSP-EIS(4,5) method to the implicit non-SSP iEIS+(4,5) method  presented in Section 4.2.
Clearly, the maximal rise in total variation of the implicit non-SSP method (dotted cyan line) is large for any $\dt$, while the
maximal rise in total variation from the eSSP-EIS(4,5) method (blue solid line) remains very small 
$(\approx 10^{-15})$ up to $\lambda \approx 1.16$. The  maximal difference between the total variation of the
solution and the post-processed solution (red dashed line) also remains small $(\approx 10^{-14})$.

We see that the difference between the total variation of the solution at the  final time and the postprocessed solution
is minimal for almost all the values of $\dt$ for which the maximal rise in total variation remains bounded.
We observe a jump in total variation of the post-processed solution occurs for a slightly smaller $\dt$ than 
the value at which we observe the jump in the total variation of the un-processed solution $u^n$.
Although the method itself was designed to be SSP, the post-processor is only designed to
extract a higher order solution but not to preserve the strong stability properties. 
This is not a concern because preserving the nonlinear stability properties is generally only
important for the stability of the time evolution: once we reach the final time solution these properties
are no longer needed.  Although we do not expect the  post-processor to preserve the nonlinear stability 
properties that the time-evolution does preserve, it is pleasant to see that for this case it does indeed do so
for most relevant values of $\dt$.

\subsection{Comparison of implicit schemes}
In Section 4.2 we presented implicit EIS schemes that can be post-processed to attain higher order.
Here we demonstrate that these methods attain the design order of convergence on several standard test 
cases, and show that although these methods suffer from order reduction (as expected from implicit schemes
that have lower stage order than overall order) the size of the errors is still small.

%
%

\begin{table}[h] \label{implicitAdvDiff} {
\footnotesize
\begin{center}
\begin{tabular}{|c|cc|cc||cc|cc|} \hline
& \multicolumn{2}{c|}{at final time}& \multicolumn{2}{c|}{post-processor} & \multicolumn{2}{c|}{at final time}&
\multicolumn{2}{c|}{post-processor } \\ \hline
& \multicolumn{4}{|c|}{iEIS+(2,3)} & \multicolumn{4}{|c|}{iEIS+(2,3)$_p$}   \\ 
M & error &   order  & error & order& error &   order  & error &   order\\ \hline
100 &  8.95$\times 10^{-4}  $   &      --     &    8.49$\times 10^{-5}  $   &      --   
       &  4.48$\times 10^{-3}  $   &      --     &    3.20$\times 10^{-4}  $   &      --  \\
150 &  3.95$\times 10^{-4}  $   & 2.02     &    2.50$\times 10^{-5}  $   & 3.01  
&  2.04$\times 10^{-3}  $   & 1.94     &    9.79$\times 10^{-5}  $   & 2.92 \\
200 &  2.21$\times 10^{-4}  $   & 2.02     &    1.05$\times 10^{-5}  $    & 3.01 
&  1.16$\times 10^{-3}  $   & 1.96     &    4.20$\times 10^{-5}  $    & 2.95  \\
250 &  1.41$\times 10^{-4}  $   & 2.01     &    5.38$\times 10^{-6}  $    & 3.01  
       &  7.95$\times 10^{-4}  $   & 1.97     &    2.17$\times 10^{-5}  $    & 2.96 \\
300 &  9.78$\times 10^{-5}  $    & 2.01    &    3.11$\times 10^{-6}  $    & 3.01 
       & 5.23$\times 10^{-4}  $   & 1.98    &    1.26$\times 10^{-5}  $    & 2.97   \\ \hline	
& \multicolumn{4}{|c|}{iEIS+(3,4)$_p$} & \multicolumn{4}{|c|}{iEIS+(4,5)$_p$}  \\
100 &  3.29$\times 10^{-5}  $   &      --     &     4.33 $\times 10^{-6}  $   &      --  
       &  8.32$\times 10^{-7}  $   &      --     &    5.13$\times 10^{-8}  $   &      --  \\
150 &  9.51$\times 10^{-6}  $   & 3.06     &     8.60$\times 10^{-7}  $   & 3.99    
       &  1.64$\times 10^{-7}  $   & 4.01     &     7.24$\times 10^{-9}  $   & 4.83 \\ 
200 &  3.96$\times 10^{-6}  $   & 3.04     &     2.73$\times 10^{-7}  $    & 3.99   
       &  5.17$\times 10^{-8}  $    &4.00  &        1.78$\times 10^{-9}  $    & 4.88 \\   
250 &  2.01$\times 10^{-6}  $   & 3.03     &     1.12$\times 10^{-7}  $    & 3.99   
       &  2.12$\times 10^{-8}  $   & 4.00   &       5.94$\times 10^{-10}  $    & 4.91 \\ 
300 &  1.16$\times 10^{-6}  $    & 3.03    &     5.40$\times 10^{-8}  $    & 3.99   
       &  1.02$\times 10^{-8}  $    & 4.00   &      2.42$\times 10^{-10}  $    & 4.93   \\ \hline	
 \end{tabular}
 \end{center} }
   \caption{Implicit solvers on Advection diffusion problem. The step-size is $\dt = 1/M$ where $M$ is given in the table. 
   Four methods from Section 4.1 are tested. The reference solution is computed by {\sc Matlab}'s ode45 subroutine.}
    \end{table}

\noindent{\bf Advection-diffusion problem:}
We repeat the advection-diffusion problem above and evolve the ODE
\[ y' = \left( -D_x   + \frac{1}{10} D_x^2 \right) y,\]
where $D_x$ is  the first order Fourier differentiation matrix and $D_x^2$ is the second order 
Fourier differentiation matrix based on $N=41$ points in space.
We use the implicit methods iEIS+(2,3),
 iEIS+(2,3)$_p$, iEIS+(3,4)$_p$, and  iEIS+(4,5)$_p$  to evolve this problem to the final time 
 $T_f=1.0$ and postprocess the solution at the final time as described in Section \ref{sec:postproc}. 
 Note that $\dt= \frac{1}{M}$ can be much larger here than when using explicit methods.
 We compute the  reference solution using {\sc Matlab}'s ode45 subroutine.
The numerical tests confirm that we observe the order of accuracy predicated by the EIS theory for the solution
and the post-processed solution.

\noindent{\bf Prothero--Robinson problem:}
This problem has been used to reveal the error reduction phenomenon that affects implicit methods.
We test our implicit methods on the non-autonomous ODE
\[ \frac{ dy}{dt}  = -a (y-sin(t)) + cos(t) \]
with initial condition $y^0 = 0$.  
We use the values $a=10$ and $a=1000$, to show the order reduction phenomenon.
We run this problem to final time $T_f=1.0$ using the iEIS+(s,p)$_p$ methods.
Note that this problem has the solution $y = sin(t)$ regardless of the value of $a$,
making the comparison easy.

Figure \ref{ProtheroRobinson} (left)  shows that the order of convergence for the case $a=10$ is 
close to the $p+1$ design order  without post-processing and the enhanced $p+2$ with post-processing. 
In contrast, the right graph shows that when $a=1000$ the convergence rate without post-processing is just
what is expected from a  truncation error analysis, while after  post-processing there is improvement, 
but the order reduction  is still apparent. 
However, it is important to note that the magnitude of the errors is {\em smaller} 
in the case $a=1000$ than when $a=10$. In this case, we observe that the order reduction phenomenon
is apparent but does not result in an increase of the errors. 

\begin{figure}
    \centering
          \includegraphics[width=.51\textwidth]{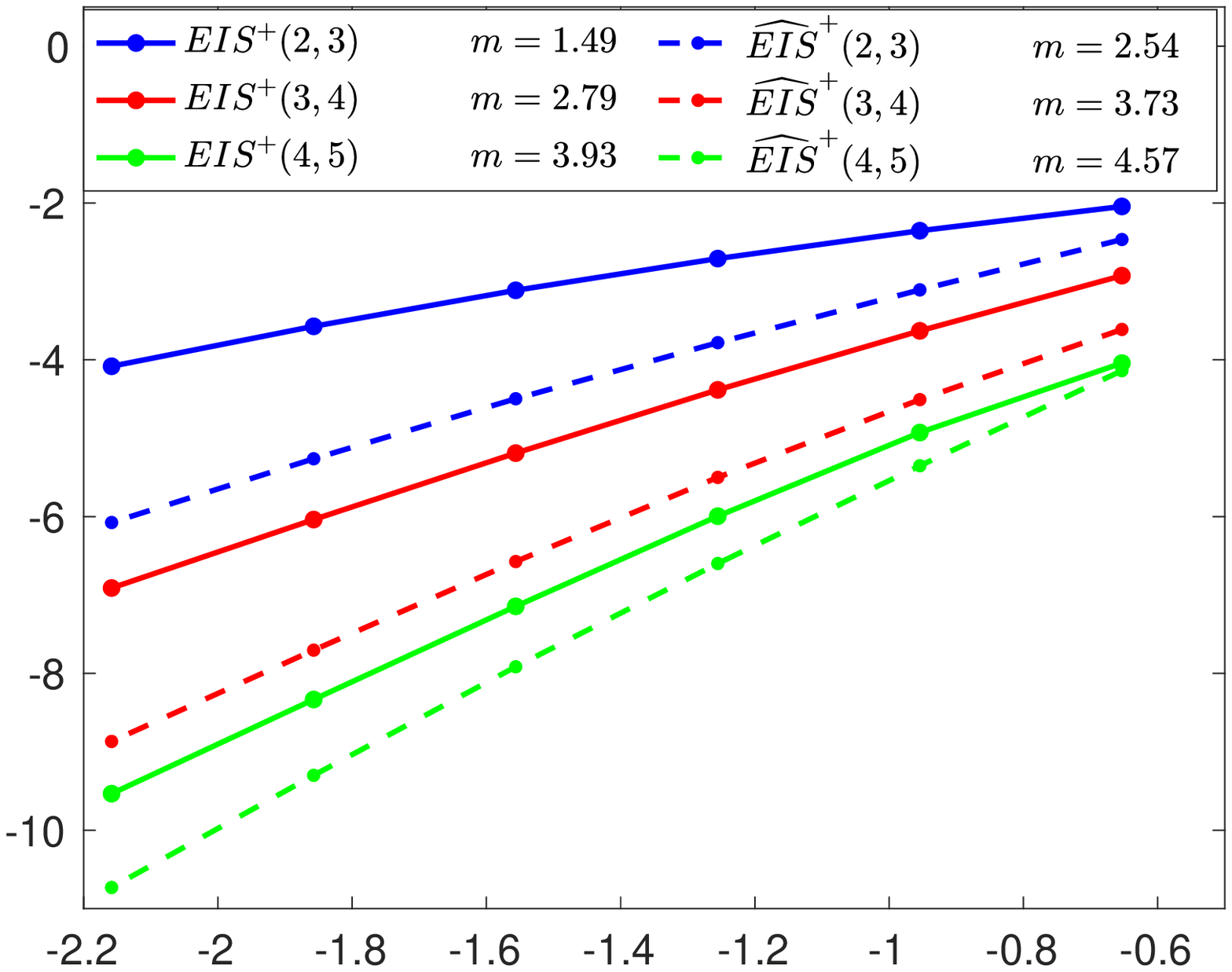} \hspace{-.2in}
            \includegraphics[width=.51\textwidth]{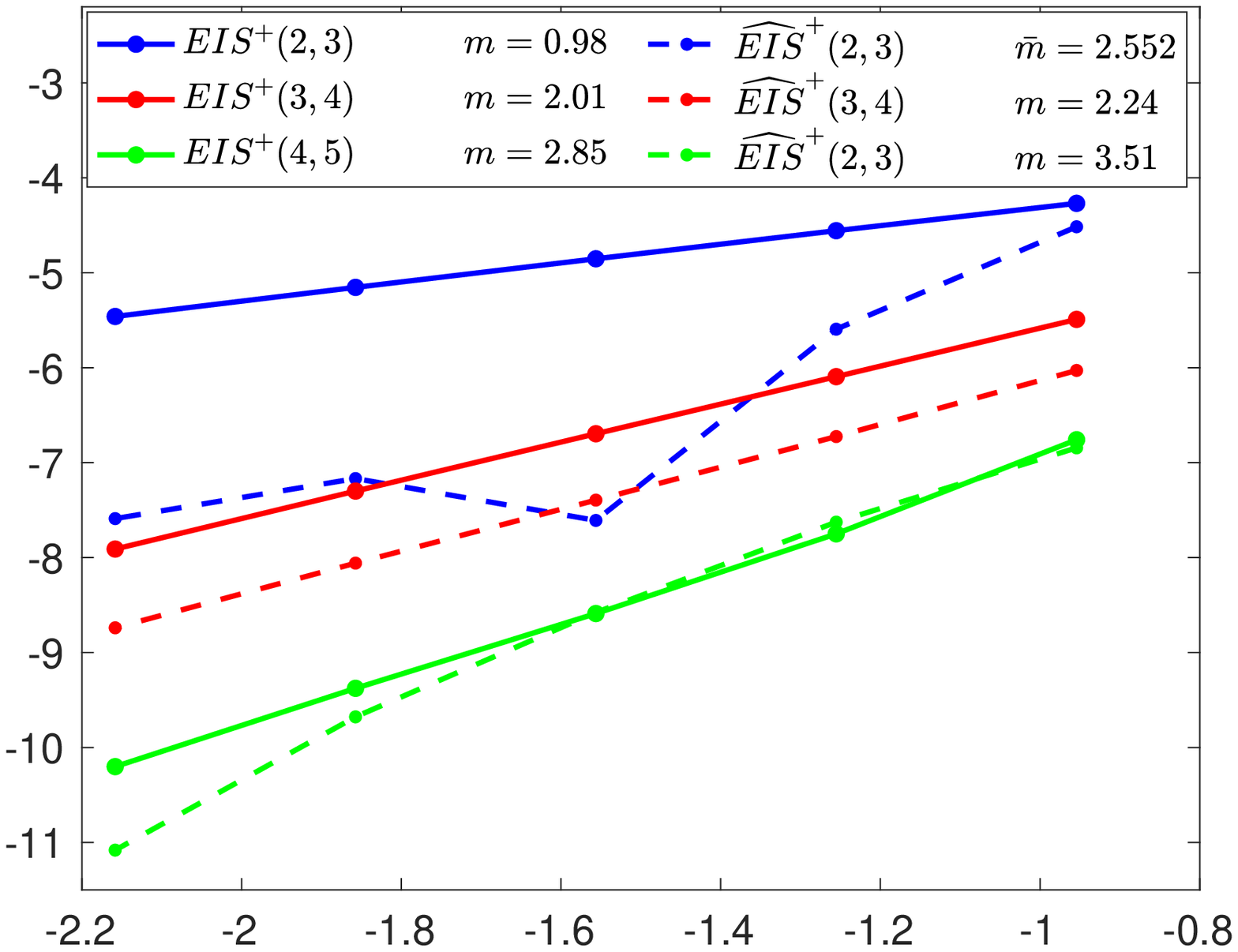}
           \caption{Convergence test on the Prothero Robinson problem. The x-axis has $log_{10}$
           of the time-step while the y-axis has $log_{10}$
           of the errors. On the left is $a=10$ and on the right $a=1000$.
            \label{ProtheroRobinson}  }
\end{figure}

\section{Conclusions}  \label{sec:conclusions}
In this work we extended the  EIS framework in \cite{EISpaper1} to explicit methods that use function 
evaluations of newly computed values and
to implicit methods. More significantly, we presented additional conditions on the coefficients of the method 
that allow the final solution to be post-processed in order to extract higher order accuracy. 
The new EIS conditions \eqref{EIS2conditions} not only control the growth of the errors (as we showed in \cite{EISpaper1}) 
but also allow us to precisely define the leading error term. Knowing the form of the 
leading error term we can compute the post-processor defined in \ref{sec:postproc}, and apply it to the solution to extract a solution 
that is two orders of accuracy higher.  

We used the new EIS+ framework to formulate an optimization code in {\sc Matlab}
to find methods that satisfy the order and the EIS+ conditions. We presented some of these
EIS+ methods and their stability regions and we tested them on sample problems to show their convergence properties.
We confirmed that the numerical solutions coming from these methods are indeed
 one  order higher than expected from a truncation error analysis, and two orders higher when post-processed. 
In future work, we plan to consider other conditions on $\mD, \mA, \mR$ that allow the error inhibiting behavior to occur,
 to extend this approach to multi-derivative and implicit-explicit methods, and to create variable step-size methods with
 EIS properties.

\appendix

\section{An alternative error analysis} 
\label{sec:appendixA}


In this section we provide an alternative proof for the accuracy of the proposed schemes. This proof is similar to the the one given in \cite{EISpaper1}.  This proof directly tracks the dynamics of the error, 
rather than tracking the behavior of the solution, as we did in Section 3.2.

We write the scheme as
\begin{equation} \label{App_A_10}
\left( I -  \Delta t \mR F  \right)  V^{n+1}  \,= \, 
\left(  \mD  +  \Delta t \mA F \right) V^n \, ,
\end{equation}
where we assume, as we did above, that the scheme  \eqref{EISmethod}  (equivalently \eqref{App_A_10} )
 is zero--stable, so that the operator $\left( I -  \Delta t \mR F  \right)^{-1} $ is bounded. 
In order to evaluate the operator $\left( I -  \Delta t \mR F  \right)$ we use Lemma \ref{lemma1} to obtain:
\begin{eqnarray}\label{App_A_20}
F(V^{n+1}  )   &= & F(U^{n+1} + E^{n+1} ) \nonumber \\
 &= & F(U^{n+1} )  + F_y^{n+1}  E^{n+1}   + O(\dt)  O\left( \|E^{n+1} \| \right) +  O\left( \|E^{n+1} \|^2 \right)  \nonumber \\
&= &  F(U^n)  + F_y^n E^{n+1}  +  O(\dt)  O\left( \|E^{n+1}\| \right) \,,
\end{eqnarray}
where, $F_y^n = F_y(u(t_n))$ and $F_y^{n+1} = F_y(u(t_{n+1} ))$.

We now subtract \eqref{localerrors}  from \eqref{App_A_10}, 
and use \eqref{App_A_20} to obtain an equation for the relationship between 
the errors at two successive time-steps:
\begin{equation} \label{App_A_30}
\left( I -  \Delta t \mR F_y^n  +  O(\dt^2)  \right )  E^{n+1}  \,= \, 
\left(  \mD  +  \Delta t \mA F_y^n  +  O(\dt^2) \right) E^n  + \ste^{n+1} \, .
\end{equation}

For the accuracy analysis we assume that $ |F_y^n|= O(1)$, and therefore
 $ \|  \dt  \mR F_y^n\| = O(\dt) \ll 1$. This observation is then used for approximating 
\begin{equation} \label{App_A_40}
 \left( I -  \Delta t \mR F   +  O(\dt^2)  \right)^{-1} \,= \, I +  \Delta t \mR F_y^n  + O(\dt^2)  \,.
\end{equation}
We then plug this into \eqref{App_A_30} to obtain a linear recursion relation for the error:
\begin{eqnarray} \label{App_A_50}
E^{n+1} &= &\left( I -  \Delta t \mR F_y^n  +  O(\dt^2) \right)^{-1} 
\left[ \left(  \mD  +  \Delta t \mA F_y^n +  O(\dt^2) \right) E^n  + \ste^{n+1} \right ], \nonumber \\
&= &   \left(  \mD  +  \Delta t F_y^n  \left( \mR  \mD+\mA \right)+  O(\dt^2) \right) E^n +  \left( I +  \Delta t F_y^n \mR  +  O(\dt^2) \right) \ste^{n+1} \nonumber \\  
&\equiv & \hat{Q}^n E^n + \dt T_e^n \,,
\end{eqnarray}
where
\begin{equation}\label{App_A_55}
 T_e^n = \Delta t^{p}  \ste^{n+1}_{p+1} +  \Delta t^{p+1} \left ( \ste^{n+1}_{p+2} + F_y^n \mR \ste^{n+1}_{p+1} \right ) + O\left ( \Delta t^{p+2} \right )
\end{equation}

\begin{lem}  \label{App_A_lemma1}
The equation which governs the dynamics of $E^n$,  \eqref{App_A_50},  is essentially linear in the sense that there is a time interval, $0 \le t \le T$, 
such that  the nonlinear terms are of higher order, and thus much smaller, than the leading terms in the equation.
\end{lem}

\begin{proof} 
It is assumed that the initial numerical condition, $V^0$, which is derived by the initial condition of the ODE and other schemes, is either accurate to machine precision or, at least accurate as desired.  Thus, at $n=0$, the scheme is essentially linear.

By assumption, the scheme is zero--stable, therefore 
  \begin{displaymath}
\| \left(  \mD  +  \Delta t F_y^n  \left( \mR \mD +\mA \right)+  O(\dt^2) \right)\| \le 1+c \dt \le \exp({c \dt })
  \end{displaymath}
and, due to the boundedness of $F_y^n$ and $\mR$
 \begin{displaymath}
\| \left( I +  \Delta t F_y^n \mR  +  O(\dt^2) \right)  \| \le 2 \,.
  \end{displaymath}
Therefore, 
 \begin{displaymath}
\| E^n  \| \le 2 \frac{\exp({c \, t_n })-1}{c \dt}\max_{0\le\nu\le n}\|\ste^{\nu} \|\,.
  \end{displaymath}
This estimate holds as long as $\| E^n  \|^2 \ll  O(\dt^2)  \max_{0\le\nu\le n}\|\ste^{\nu}  \| \, .$\\

\end{proof}

As the error equation satisfies an iterative process of the form \eqref{App_A_50}, we state the 
following Lemma that we will use later to understand the dynamics of the growth in the error.
\begin{lem}  \label{App_A_lemma2}

Given an iterative process of the form
\begin{equation}\label{App_A_90}
W^{n+1} \,=\, Q^n W^n + \dt F(W^n, t_n)
\end{equation}
where $Q^n$ is a linear operator, the 
Discrete Duhamel principle states that  
\begin{equation} \label{App_A_100}
W^{n} \, = \,   \prod_{\mu=0}^{n-1} Q^{\mu}  W^0  \,+\,  \dt \, \sum_{\nu=0}^{n-1}  
\left( \prod_{\mu=\nu+1}^{n-1} Q^{\mu}  \right)  F(W^{\nu}, t_{\nu}) \;.
\end{equation}
\end{lem}
This is  Lemma 5.1.1 in  \cite{gustafsson1995time} and the proof  is given there.
\bigskip

Using the observation that Equation  \eqref{App_A_50}, which governs the dynamics of $E^n$, 
 is essentially linear, we can use the Discrete
Duhamel principle, \eqref{App_A_100} to calculate $E^n$.
\begin{equation} \label{App_A_110}
E^{n} \, = \,   \prod_{\mu=0}^{n-1} \hat{Q}^{\mu}  E_0  \,+\,  \dt \, \sum_{\nu=0}^{n-1}  
\left( \prod_{\mu=\nu+1}^{n-1} \hat{Q}^{\mu}  \right)  T_e^{\nu} \;.
\end{equation}

The first term in \eqref{App_A_110} is negligible because we assume that the initial error can be made arbitrarily small. In order to evaluate the second term we divide the sum into three parts:
\begin{enumerate}
\item The final term, $\nu= n-1$, is $\dt T_e^{n-1}$ which is clearly of order $\dt \| T_e^{n-1} \| = \| \ste^{n+1} \|  = O(\dt^{p+1})$. This term is the one filtered in the postprocessing stage.

\item The next term, $\nu= n-2$, is $\dt \hat{Q}^{n-1} T_e^{n-2}$.  This term can be made  of order $\dt^2 \| T_e^{n-2} \| $ provided that $\|\hat{Q}^{n-1} T_e^{n-2} \| = O\left (\dt \right) O \left(\| T_e^{n-2} \| \right )$ which is true due to \eqref{con1}: $ \mD \ste_{p+1} = 0$.

\item The rest of the sum;
\begin{eqnarray} \label{App_A_120}
\left\| \dt \, \sum_{\nu=0}^{n-3}  \left( \prod_{\mu=\nu+1}^{n-1} \hat{Q}^{\mu}  \right)  T_e^{\nu}  \right\| & = &
\left\| \dt \, \sum_{\nu=0}^{n-3}  \left( \prod_{\mu=\nu+3}^{n-1}\hat {Q}^{\mu}  \right)   \left(\hat {Q}^{\nu+2}\hat {Q}_{\nu+1} T_e^{\nu} \right) \right\|  \nonumber \\
& \leq & \dt \, \sum_{\nu=0}^{n-3}  \left\|  \prod_{\mu=\nu+3}^{n-1} \hat{Q}^{\mu}  \right\|   \left\|  \hat{Q}^{\nu+2}  \hat{Q}_{\nu+1} T_e^{\nu}  \right\| \nonumber  \\
& \leq & \dt \, \sum_{\nu=0}^{n-3}\left( 1+ c \,\dt \right)^{n- \nu -3}   \left\|  \hat{Q}^{\nu+2}  \hat{Q}^{\nu+1} T_e^{\nu}  \right\| \nonumber  \\
& = & \frac{\exp \left( c t_n\right) -1}{c}  \max_{\nu=0, . . .,  n-3}  \left\|  \hat{Q}^{\nu+2}  \hat{Q}^{\nu+1} T_e^{\nu}  \right\|  . 
\end{eqnarray}
Using the definition of the operators $\hat{Q}^\mu$ we have:
\[  \hat{Q}^{\nu+2}  \hat{Q}^{\nu+1} =  \left[  \mD^2  +  \Delta t
\left(  F_y^{\nu+2}   \left( \mR \mD+\mA \right) \mD +F_y^{\nu+1}  \mD \left( \mR \mD +\mA \right)  \right )
+  O(\dt^2) \right] \]
so that $ \hat{Q}^{\nu+2}  \hat{Q}^{\nu+1}T_e^{\nu}$ becomes
\begin{eqnarray*}
&&  \left[  \mD^2  +  \Delta t
\left(  F_y^{\nu+2}   \left( \mR \mD+\mA \right) \mD +F_y^{\nu+1}  \mD \left( \mR \mD +\mA \right)  \right )
+  O(\dt^2) \right] T_e^{\nu} \\
 &=&   \mD^2   T_e^{\nu}  +  \Delta t
 F_y^{\nu+2}   \left( \mR \mD+\mA \right) \mD   T_e^{\nu} 
+   F_y^{\nu+1}  \mD \left( \mR \mD +\mA \right)T_e^{\nu} +
  O(\dt^2)  T_e^{\nu} \\
  & = &   \Delta t^{p}  \mD^2  \ste^{\nu+1}_{p+1} +  
  \Delta t^{p+1} \mD^2  \left( \ste^{\nu+1}_{p+2} + F_y^{\nu+1} \mR \ste^{\nu+1}_{p+1} \right) \\
  & +  &   \Delta t^{p+1}  F_y^{\nu+2}  \left(  \mR \mD+\mA \right) \mD   \ste^{\nu+1}_{p+1} +   
 \Delta t^{p+1}   F_y^{\nu+1}  \mD \left( \mR \mD +\mA \right)  \ste^{\nu+1}_{p+1} +   
 O(\dt^{p+2})  \\
 & = &   \Delta t^{p}  \left( \mD  +  
 \Delta t  F_y^{\nu+2}  \left(  \mR \mD+\mA \right) 
 + \dt F_y^{\nu+1}  \mD \mR \right)  \mD  \ste^{\nu+1}_{p+1}  \\
  &+ &  
  \Delta t^{p+1} \mD^2  \left( \ste^{\nu+1}_{p+2} + F_y^{\nu+1} \mR \ste^{\nu+1}_{p+1} \right) +
   \Delta t^{p+1}   F_y^{\nu+1}  \mD \mA  \ste^{\nu+1}_{p+1} +   O(\dt^{p+2})  \\
  & = &   \Delta t^{p}  \left( \mD  +  
 \Delta t  F_y^{\nu+2}  \left(  \mR \mD+\mA \right) 
 + \dt F_y^{\nu+1}  \mD \mR \right)  \mD  \ste^{\nu+1}_{p+1}  \\
  &+ &    \Delta t^{p+1} \left( F_y^{\nu+1} \mD^2  \mR   +  F_y^{\nu+1}  \mD \mA  \right) \ste^{\nu+1}_{p+1} 
  +  \Delta t^{p+1} \mD^2  \ste^{\nu+1}_{p+2} +   O(\dt^{p+2}) . 
\end{eqnarray*}
Using the fact that in our case $\mD^2 = \mD$ and that $F_y^{\nu+1}  = F_y^{\nu}  +O(\dt)$ we obtain
\begin{eqnarray*}
 \hat{Q}^{\nu+2}  \hat{Q}^{\nu+1}T_e^{\nu} & = & 
 \Delta t^{p}  \left( \mD  +  
 \Delta t  F_y^{\nu+2}  \left(  \mR \mD+\mA \right) 
 + \dt F_y^{\nu+1}  \mD \mR \right)  \mD  \ste^{\nu+1}_{p+1}  \\
  &+ &    \Delta t^{p+1}  F_y^{\nu+1}  \mD  \left(   \mR   +   \mA  \right) \ste^{\nu+1}_{p+1} 
  +  \Delta t^{p+1} \mD  \ste^{\nu+1}_{p+2} +   O(\dt^{p+2}) .
\end{eqnarray*}
The first term and third terms disappear because  of \eqref{con1} and \eqref{con2}
\[\mD \ste_{p+1} = 0 \; \; \; \; \; \mbox{and} \; \; \; \; \; \mD \ste_{p+2} =0. \] 
The second term is eliminated by \eqref{con3} 
\[ \mD  \left(   \mR   +   \mA  \right) \ste_{p+1}  = 0 .\]
So that 
\[  \hat{Q}^{\nu+2}  \hat{Q}^{\nu+1}T_e^{\nu}  = O(\dt^{p+2})  = O(\dt^{2})  O(\|T^\nu_e\|).\]

Putting this all back together we see that 
\[
E^{n} \, = \,  \dt T^{n-1}_e + O(\dt^{2})  O(\|T^{n-2}_e\|) +O(\dt^{2})  O(\|T^{n}_e\|).
\]

\end{enumerate}

\section{From scalar to vector notation} 
\label{sec:appendixB}
In Section \ref{EISppTheory}  we developed the theory for EIS methods with post-processing.
In that section we used, for simplicity, scalar notation. In this appendix we show in detail how 
the vector case is similarly developed. For ease of explanation, we consider the equation
\[ \left( \begin{array}{l}
u(t) \\
\omega(t) \\
\end{array} \right)_t 
= \left( \begin{array}{l}
f(u(t),\omega(t)) \\
g(u(t),\omega(t)) \\
\end{array} \right).
\]
This is only a vector of two components, but it will be clear from this explanation why the theory developed for the scalar
case is sufficient for vectors of any number of components.

The method still has a similar  form to  \eqref{EISmethod}
\begin{equation} 
V^{n+1} = \mD \otimes V^n +  \Delta t \mA  \otimes F( V^n) +  \Delta t \mR  \otimes F( V^{n+1}) , 
\end{equation}
 but here
$V^n$ is a  vector of length $2s$ that contains the numerical solutions
at times  $\left( t_n+ c_j \Delta t \right)$ for $j=1,\ldots,s$:
\begin{equation*}
V^n = \left(v(t_n+ c_1 \Delta t))  , \ldots ,v(t_{n} + c_{s} \Delta t ) ,
w(t_n+ c_1 \Delta t)) , \ldots ,w(t_{n} + c_{s} \Delta t ) 
 \right )^T ,
\end{equation*}
which are approximations to the exact solutions 
\begin{equation*}
U^n = \left(u(t_n+ c_1 \Delta t)) , \ldots ,u(t_{n} + c_{s} \Delta t ) ,
\omega(t_n+ c_1 \Delta t)) ,   \ldots ,\omega(t_{n} + c_{s} \Delta t ) 
 \right )^T .
\end{equation*}
Correspondingly, we define
\begin{equation*}
F(U^n )   =   \left( \begin{array}{c}
f(v(t_n+ c_1 \dt ) ,  w(t_n+ c_1 \dt ) )\\
f(v(t_n+ c_2 \dt ) ,  w(t_n+ c_2 \dt ) )\\
\vdots  \\
f(v(t_n+ c_s \dt ) ,  w(t_n+ c_s \dt ) )\\
g(v(t_n+ c_1 \dt ) ,  w(t_n+ c_1 \dt ) )\\
g(v(t_n+ c_2 \dt ) ,  w(t_n+ c_2 \dt ))\\
\vdots  \\
g(v(t_n+ c_s \dt ) ,  w(t_n+ c_s \dt ))\\
 \end{array} \right).
\end{equation*}
We see that the vectors $U^n$ and $V^n$ are simply concatenations of the two solution vectors,
and the vector function $F$ contains two different functions, each operating on the two different solution vectors.

We use the Kronecker product $\otimes $ notation to mean that the $s\times s$ matrices which operate separately on the first
set of $s$ elements and on the second set of $s$ elements. Thus, the local truncation error and approximation errors are now defined
exactly as in \eqref{localerrors} but the Taylor expansions give us
\[
 \Delta t \; LTE^n = \ste^n =  \sum_{j=0}^{\infty} \ste^n_j  \Delta t^{j}   =
 \left(
 \begin{array}{l}
 \sum_{j=0}^{\infty} \ste_j  \Delta t^{j}  \left. \frac{d^j u} {dt^j} \right|_{t=t_n}   \\
 \sum_{j=0}^{\infty} \ste_j  \Delta t^{j}  \left. \frac{d^j \omega} {dt^j} \right|_{t=t_n}   \\
 \end{array} \right)
\]

The observation in Lemma \ref{lemma1} can be re-cast in this vector form as follows:
given smooth enough functions $f$ and $g$,
\begin{eqnarray*}
F(U^n + E^n)   &=&   \left( \begin{array}{c}
f(u(t_n+ c_1 \dt ) + e(t_n+ c_1 \dt ),  \omega(t_n+ c_1 \dt )+\epsilon(t_n+ c_1 \dt ) )\\
f(u(t_n+ c_2 \dt ) + e(t_n+ c_2 \dt ),  w(t_n+ c_2 \dt )+\epsilon(t_n+ c_2 \dt ) )\\
\vdots  \\
f(u(t_n+ c_s \dt ) + e(t_n+ c_s \dt ),   \omega(t_n+ c_s \dt )+\epsilon(t_n+ c_s \dt ) )\\
g(u(t_n+ c_1 \dt ) + e(t_n+ c_1 \dt ),   \omega(t_n+ c_1 \dt )+\epsilon(t_n+ c_1 \dt ) )\\
g(u(t_n+ c_2 \dt ) + e(t_n+ c_2 \dt ),   \omega(t_n+ c_2 \dt )+\epsilon(t_n+ c_2 \dt ) )\\
\vdots  \\
g(u(t_n+ c_s \dt ) + e(t_n+ c_s \dt ),   \omega(t_n+ c_s \dt )+\epsilon(t_n+ c_s \dt ) )\\
 \end{array} \right)
 \end{eqnarray*}
 can be Taylor expanded term by term, so that  the first $s$ terms are 
 \begin{eqnarray*} 
\left(F(U^n + E^n)\right)_j   & = & 
 f(u(t_n+ c_j \dt ) + e(t_n+ c_j \dt ),   \omega(t_n+ c_j \dt )+\epsilon(t_n+ c_j \dt ) ) \\
& = & f(u(t_n+ c_j \dt ) , w(t_n+ c_j \dt ) )  \\
&+&  e(t_n+ c_j \dt ) f_u( u(t_n+ c_j \dt ) ,  \omega(t_n+ c_j \dt ) ) \\
& + &   \epsilon (t_n+ c_j \dt ) f_ \omega( u(t_n+ c_j \dt ) ,  \omega(t_n+ c_j \dt ) ) + O(\|e^n\|^2, \|\epsilon^n\|^2,\|e^n \epsilon^n\|). \\
& = & f^{n+ c_j} +  e^{n+ c_j} f_u^n +   C \dt e^{n+ c_j} +   \epsilon^{n+ c_j} f_\omega^n + \tilde{C} \dt \epsilon^{n+ c_j} 
+ O(\dt^{p+3}). 
  \end{eqnarray*}
  and similarly the next $s$ terms are 
  \begin{eqnarray*} 
\left(F(U^n + E^n)\right)_j   & = &  g^{n+ c_j} +  e^{n+ c_j} g_u^n +   \hat{C} \dt e^{n+ c_j} +   \epsilon^{n+ c_j} g_\omega^n 
+ \overline{C} \dt \epsilon^{n+ c_j}  + O(\dt^{p+3}).  
  \end{eqnarray*}

The main occurrence of the derivative term $F_y$ is in the second step of the proof. 
We can look at the first half of the vector, denoting it $ V^{k+1}_{[1:s]} $:
  \begin{eqnarray*}
 V^{k+1}_{[1:s]} & = & \mD V^k_{[1:s]} + \Delta t \mA F(V^k)_{[1:s]} + \Delta t \mR   F(V^{k+1})_{[1:s]}  \\ 
& = & \mD   \left(U^k + E^k \right)_{[1:s]} +  \Delta t \mA  F\left(U^k + E^k  \right)_{[1:s]}+ \Delta t \mR  F\left(U^{k+1}+   E^{k+1} \right)_{[1:s]} \\
  & = &  \mD  U^k_{[1:s]} + \mD  E^k_{[1:s]} +  \Delta t \mA   F(U^k)_{[1:s]}  +   \Delta t \mR F(U^{k+1})_{[1:s]} \\
  & +&  \Delta t  \left(  f_u^k   \mA E^k_{[1:s]} + f_\omega^k   \mA  E^k_{[s+1:2s]}) \right) 
  + \dt^2 \mA \left( C E^{k}_{[1:s]}   + \tilde{C}  E^{k}_{[s+1:2s]}  \right)   \\
     &+&   \Delta t  \left(  f_u^{k+1}   \mR E^{k+1}_{[1:s]} + f_\omega^{k+1}   \mR  E^k_{[s+1:2s]}) \right) 
+ \dt^2 \mR \left( C E^{k+1}_{[1:s]}   + \tilde{C}  E^{k+1}_{[s+1:2s]}  \right)   +O(\dt^{p+3})
   \end{eqnarray*}
We can write a similar formula (but using $g$) for the second half of the vector.
Despite the additional terms, this is similar to the proof in Section 3, because  the additional terms are of similar form.
 The top half and bottom half of the error vectors $E^k$ a nd $E^{k+1}$
each contain the vectors $\ste_{p+1}$ and $\ste_{p+2}$, and the error inhibiting conditions act on these parts. Thus the proof goes through
seamlessly.

\bigskip

{\bf Acknowledgment.} 
This publication is based on work supported by  AFOSR grant FA9550-18-1-0383. ONR-DURIP Grant N00014-18-1-2255. A part of this research is sponsored by the Office of Advanced Scientific Computing Research; US Department of Energy, and was performed at the Oak Ridge National Laboratory, which is managed by UT-Battelle, LLC under Contract no. De-AC05-00OR22725. This manuscript has been authored by UT-Battelle, LLC, under contract DE-AC05-00OR22725 with the US Department of Energy. The United States Government retains and the publisher, by accepting the article for publication, acknowledges that the United States Government retains a non-exclusive, paid-up, irrevocable, world-wide license to publish or reproduce the published form of this manuscript, or allow others to do so, for United States Government purposes.


\begin{thebibliography}{10}

\bibitem{AllenIsaacson}
M.~B. Allen and E.~L. Isaacson, {\em Numerical analysis for applied
  science}, John Wiley \& Sons, 1998.

\bibitem{msrk}
C. Bresten, S. Gottlieb, Z. Grant, D. Higgs, D.I. Ketcheson, and A. Nemeth, 
{\em Explicit strong stability preserving multistep Runge-Kutta methods},
  Mathematics of Computation {\bf 86} (2017) 747--769.

\bibitem{butcher1993a}
J.~C. Butcher, {\em Diagonally-implicit multi-stage integration method},
  Applied Numerical Mathematics {\bf 11} (1993), 347--363.
  
\bibitem{butcher2008numerical}
J.~C. Butcher, {\em Numerical methods for ordinary differential equations},
  John Wiley \& Sons, 2008.

\bibitem{ditkowski2015high}
A~Ditkowski, {\em High order finite difference schemes for the heat equation
  whose convergence rates are higher than their truncation errors}, Spectral
  and High Order Methods for Partial Differential Equations ICOSAHOM 2014,
  Springer, 2015, pp.~167--178.

\bibitem{EISpaper1} 
A. Ditkowski and S. Gottlieb, 
{\em Error Inhibiting Block One-step Schemes for Ordinary Differential Equations},
  Journal of Scientific Computing  {\bf 73(2)} (2017) 691--711.

\bibitem{DitkowskiICERM}
A. Ditkowski
{\em Error inhibiting schemes for differential equations},
lecture given at ICERM on August 2018,
\url{ https://icerm.brown.edu/video_archive/?play=1669}.

\bibitem{EISgithub}
S. Gottlieb, Z.J. Grant, A. Ditkowski,
{\em Explicit and Implicit EIS methods with post-processing}, (2019),
GitHub repository,
\url{https://github.com/EISmethods/EISpostprocessing}.

\bibitem{gustafsson1995time}
B. Gustafsson, H.-O. Kreiss, and J. Oliger, {\em Time dependent
  problems and difference methods}, vol.~24, John Wiley \& Sons, 1995.


\bibitem{HairerNorsettWanner}
E. Hairer, G. Wanner, and S. P. Norsett,
{\em  Solving Ordinary Differential Equations I: Nonstiff Problems},
Springer Series in Computational Mathematics,
Springer-Verlag Berlin Heidelberg (1993).


\bibitem{SSPbook2011}
J. S. Hesthaven, S. Gottlieb, D. Gottlieb, 
{\em Spectral Methods for Time Dependent Problems. Cambridge Monographs 
on Applied and Computational Mathematics (No. 21) },
Cambridge University Press (2006). 

\bibitem{IsaacsonKeller}
E. Isaacson and H. Keller, {\em Analysis of numerical methods},
  Dover Publications, Inc, 1994.
  
  \bibitem{JackiewiczBook} 
Z. Jackiewicz, {\em General linear methods for ordinary differential
  equations}, John Wiley \& Sons, 2009.
  
\bibitem{Kulikov2009}
G.Yu. Kulikov,
{\em On quasi-consistent integration by Nordsieck methods,}
Journal of Computational and Applied Mathematics {\bf 225} (2009) 268--287.

\bibitem{KulikovWeiner2010}
G. Yu. Kulikov and R. Weiner,
{\em  Variable-Stepsize Interpolating Explicit Parallel Peer Methods with Inherent Global Error Control,}
SIAM Journal on Scientific Computing, {\bf 32(4)} (2010) 1695--1723. 


\bibitem{KulikovWeiner2018}
G.Yu. Kulikov and R. Weiner,
{\em Doubly quasi-consistent fixed-stepsize numerical integration
of stiff ordinary differential equations with implicit two-step
peer methods,} 
Journal of Computational and Applied Mathematics,
{\bf 340} (2018) 256--275.

\bibitem{lax1956survey}
P.~D. Lax and R.~D. Richtmyer, {\em Survey of the stability of linear
  finite difference equations}, Communications on pure and applied mathematics
  {\bf 9} (1956), no.~2, 267--293.

\bibitem{quarteroni2010numerical}
A. Quarteroni, R. Sacco, and F. Saleri, {\em Numerical
  mathematics}, vol.~37, Springer Science \& Business Media, 2010.


 \bibitem{Skeel1978}
R.D. Skeel, 
{\em Analysis of fixed-stepsize methods,} 
 SIAM Journal on Numerical Analysis
 {\bf 13} (1976) 664--685.

\bibitem{Sewell}
G. Sewell, {\em The numerical solution of ordinary and partial
  differential equations}, World Scientific, 2015.

 
\bibitem{SoleimaniWeiner2017}
B. Soleimani, R. Weiner
{\em A class of implicit peer methods for stiff systems},
Journal of Computational and Applied Mathematics 
{\bf 316} (2017) 358--368

\bibitem{Suli2003}
E. Suli and D.F. Mayers, 
{\em An Introduction to Numerical Analysis},
Cambridge University Press, Cambridge, 2003.

 
\bibitem{KulikovWeiner2012}
 R. Weiner,  G.Yu. Kulikov, and H. Podhaiskya,
{\em Variable-stepsize doubly quasi-consistent parallel explicit peer methods
with global error control,}
 Applied Numerical Mathematics {\bf 62}  (2012) 1591--1603.

\bibitem{WeinerSchmitt2009}
R. Weiner, B. A. Schmitt, H.Podhaiskya, and S. Jebensc,
{\em Superconvergent explicit two-step peer methods,}
 Journal of Computational and Applied Mathematics
 {\bf 223} (2009) 753--764.

\end{thebibliography}
\end{document}